\documentclass[a4paper]{article}
\usepackage[margin=0.85in,footskip=0.25in]{geometry}
\usepackage{longtable,float,hyperref,xcolor,amsmath,amsxtra,amssymb,latexsym,amscd,amsthm,amsfonts,graphicx}
\usepackage{adjustbox}
\usepackage{caption}
\usepackage{multirow}
\numberwithin{equation}{section}
\allowdisplaybreaks
\usepackage{fancyhdr}
\date{}
\usepackage{imakeidx}
\makeindex[columns=2, title=Alphabetical Index, 
           options= -s index.ist]
\def\R{\mathbb{R}}
\theoremstyle{plain}
\newtheorem{thm}{Theorem} % reset theorem numbering for each section

\theoremstyle{definition}
\newtheorem{exmp}[thm]{Example} % same for example numbers
\newtheorem{rmk}[thm]{Remark}
\newtheorem{lmm}[thm]{Lemma}
\newtheorem{prop}[thm]{Proposition}

\title{Quasi-neutral Dynamics in a  Coinfection System with $N$ Strains and Asymmetries along Multiple Traits}
\author{Thi Minh Thao LE$^*$, Erida GJINI$^\perp$, Sten MADEC$^+$\\
\small{*, $^+$ Laboratory of Mathematics, University of Tours, Tours, France}\\
\small{$^\perp$ Center for Computational and Stochastic Mathematics, Instituto Superior Tecnico, Lisbon, Portugal}\\
\small{\texttt{$^*$Thi-minh-thao.Le@lmpt.univ-tours.fr, $^\perp$ erida.gjini@tecnico.ulisboa.pt, $^+$Sten.Madec@lmpt.univ-tours.fr}}\\}
\begin{document}
\maketitle

\begin{abstract}
Understanding the interplay of different traits in a co-infection system with multiple strains has many applications in ecology and epidemiology. Because of high dimensionality and complex feedbacks between traits manifested in infection and co-infection, the study of such systems remains a challenge. In the case where strains are similar (quasi-neutrality assumption), we can model trait variation as perturbations in parameters, which simplifies analysis. Here, we apply singular perturbation theory to many strain parameters simultaneously, and advance analytically to obtain their explicit collective dynamics. We consider and study such a quasi-neutral model of susceptible-infected-susceptible (SIS) dynamics among $N$ strains which vary in 5 fitness dimensions: transmissibility, clearance rate of single - and co-infection, transmission probability from mixed coinfection, and co-colonization vulnerability factors encompassing cooperation and competition. This quasi-neutral system is analyzed with a singular perturbation method through an appropriate slow-fast decomposition. The fast dynamics correspond to the embedded neutral system, while the slow dynamics are governed by an $N$-dimensional replicator equation, describing the time evolution of strain frequencies. The coefficients of this replicator system are pairwise invasion fitnesses between strains, which, in our model, are an explicit weighted sum of pairwise asymmetries along all trait dimensions. Remarkably these weights depend only on the parameters of the neutral system. Such model reduction highlights the centrality of the neutral system for dynamics at the edge of neutrality, and exposes critical features for maintenance of diversity. 
\end{abstract}

\section{Introduction} \label{sec1}
%The cooperation and competition between different virus strains to infect a pool of susceptible individuals is a common phenomenon in biology, especially in the outbreaks of infectious diseases \cite{Grover, Keeling, Roth}. These interactions, also known as co-infection, can change the stability and inherent properties of both the pathogen and the disease dynamics \cite{Balmer}. For that reason, the topic of multi-strain co-infected systems is gaining more popularity in academia and producing practical results. For example, \cite{Neil} shows that pathogen co-infection can significantly reduce vaccine effectiveness by confusing the host’s immune system. Also, \cite{Snyder, Warren} studied different co-infection systems to assist the diagnosis and treatment of infectious diseases. \\ \\

Multiple infections are ubiquitous in nature \cite{Balmer}. They may occur between pathogen strains of the same species or between different species \cite{Cobey2012,Birger,Warren}, and have implications for virulence evolution and maintenance of various polymorphisms among infectious agents \cite{van1995dynamics, mosquera1998evolution,alizon2013co,alizon2013multiple}. The importance of multiple infection for antibiotic resistance and vaccination effects in multi-strain systems has also been increasingly highlighted \cite{Lipsitch1997,birger2015potential}. Due to its inherent difficulties, multiple infection has only been tackled in a limited manner by mathematical models so far. A majority of studies focus on coexistence and competitive exclusion criteria for coinfection systems with $N=2$ or $N=3$ strains \cite{chen2017fundamental,gjini2016direct,Gjini1,pinotti2019interplay}. A few studies, using arbitrary system size, derive analytical results for any number of coinfecting strains $N$ \cite{adler1991dynamics,Madec2}. But the vast majority of $N$-strain coinfection models are entirely based on simulations \cite{Cobey2012,davies2019within}, with limited analytical insight and organic syntheses for the mechanisms of emergent dynamics.

In this article, 
%we open the `black box' of coinfection models with $N$ strains.
we uncover the subtle structure of coinfection model with $N$ strains.
 We introduce a general model to describe the population dynamics of multiple strains circulating in a host population with the possibility of co-infection. In particular, we focus on modeling the host-to-host transmission of different strains, using the SIS (susceptible - infected - susceptible) compartmental framework for endemic diseases. There are two sources of complexity in the model: i) the number of strains, which increases quadratically the dimensionality of the system, and ii) all the fitness dimensions in which the strains may vary. The latter is the main novelty of our framework. 

%For example, \cite{Birger} shows that the presence of HIV can decrease the probability of spontaneous clearance of hepatitis C virus for patients. Such varied nature in each strain’s traits leads to the variations in their interaction types and the complex dynamics of their populations \cite{Faust}. \\ \\
We present a method for approximating the solution of this SIS- N-strain co-infection system, under a quasi-neutral assumption for the strain-defining parameters. To that end, we first analyze multi-strain co-infection system with symmetric traits. Then, based on the theoretical results in \cite{Fenichel, Kuehn, Tik} and their applications to similar models in \cite{Gjini, Gjini1, Madec2}, we use the slow-fast dynamics approach and the method of multiple timescales to approximate the solution of systems with non-symmetric traits. 

Extending the foundational work in \cite{Gjini, Gjini1, Madec2}, this article studies a more general dynamic system, with perturbations in many more dimensions of variation across strains, namely transmission, clearance rates and within-host competitiveness, besides the co-colonization vulnerability parameters (Figure \ref{scheme}). The complexity of this general problem is reduced by the quasi-neutral assumption, with each parameter constrained to be close to its default value, allowing us to leverage the neutral system to approximate the quasi-neutral system.
The difficulty lies in reformulating the original system starting from a neutral component plus perturbations, in terms of                                   slow fast dynamics consisting of a fast sub-system  and a slow sub-system.  
%Our aim is to find the equations that describe the strain dynamics on the slow manifold. 
Thanks to the singular perturbation theory in \cite{Ver} and the Tikhonov's Theorem, we expect to find explicitly the emergent system which describes the slow dynamics.

More precisely,  we find how to   rewrite the original system in the form $\dfrac{dx}{d t} = \epsilon f(x,y,t,\epsilon)$ and $\dfrac{dy}{d t} = g(x,y,t,\epsilon)$ where $x$ describes the slow dynamics and $y$ the fast dynamics. Taking $\epsilon=0$ we obtain the degenerate fast system $\dfrac{dx}{d t} =0$ and $\dfrac{dy}{d t} = g(x,y,t,0)$. Under appropriate assumptions, this fast system admit a (degenerate) attractor called the {\em slow manifold} of the form $y=\phi(x,t)$. 
Then, at the {\it slow time scale} $\tau=\dfrac{t}{\epsilon}$ we obtain the {\it slow} dynamics on this slow manifold as $\dfrac{ dx}{d \tau}=f(x,\phi(x,t),\tau,0)$ that needs to be computed  explicitly.
The singular perturbation theory makes the link between this slow dynamics and the dynamics of the original system for  $0<\epsilon\ll 1$.\\

%The aim of this paper is to find the equations that describe the strain dynamics on the slow manifold.

 % which tends to the degenerate system $\dfrac{dx}{d\tau} =  f(x,y,\tau,\epsilon)$ when $\epsilon \to 0$. Given $y = \phi(x,t)$ satisfying $g(x,y,t,\epsilon) = 0$ is a uniformly asymptotically stable solution of $\dfrac{dy}{dt} = g\left(x,y,t,\epsilon\right)$. Tikhonov's theorem states that the solution of the original system approaches the solution of the degenerate system in any bounded time interval as $\epsilon \to 0$. 
Even though we have an intuition for how the final model approximation in terms of fast-slow dynamics should work, with the neutral model as the organizing centre \cite{Golu}, it is not at all obvious from the start which should be the necessary mathematical steps when multiple perturbations occur and interact at the same time between $N$ strains. In this article we uncover these steps, which ultimately lead us to a similar replicator equation to the one derived in \cite{Madec2} but now more complete because it involves variation among strains along more fitness dimensions. Indeed, we obtain an $N$ dimensional replicator equation for strain frequencies over long time in terms of their pairwise invasion fitness matrix, and this connects our multi-strain coinfection framework in an endemic setting with the work of \cite{Hofbauer} which extensively researches this well-known model, and shows its contribution to evolution and game theory. With this simplifcation, qualitative and quantitative aspects of the competitive dynamics between $N$ strains, leading to regimes of exclusion, coexistence, multi-stability, family of cycles or chaotic behavior can be investigated, and directly linked to their trait variations.
%This notion allows us to see the behavior of a two-strain model explicitly. For reference, a general coexistence problem in ecology focuses on the reason for the coexistence of different species, can be seen in \cite{Gause} - says from niche adaptation, and in \cite{MacArthur} - says from a balance of neutral processes such as immigration, speciation, and extinction. \\ \\
\begin{figure}[htb!]
	\centering
	\includegraphics[width= 0.8 \textwidth]{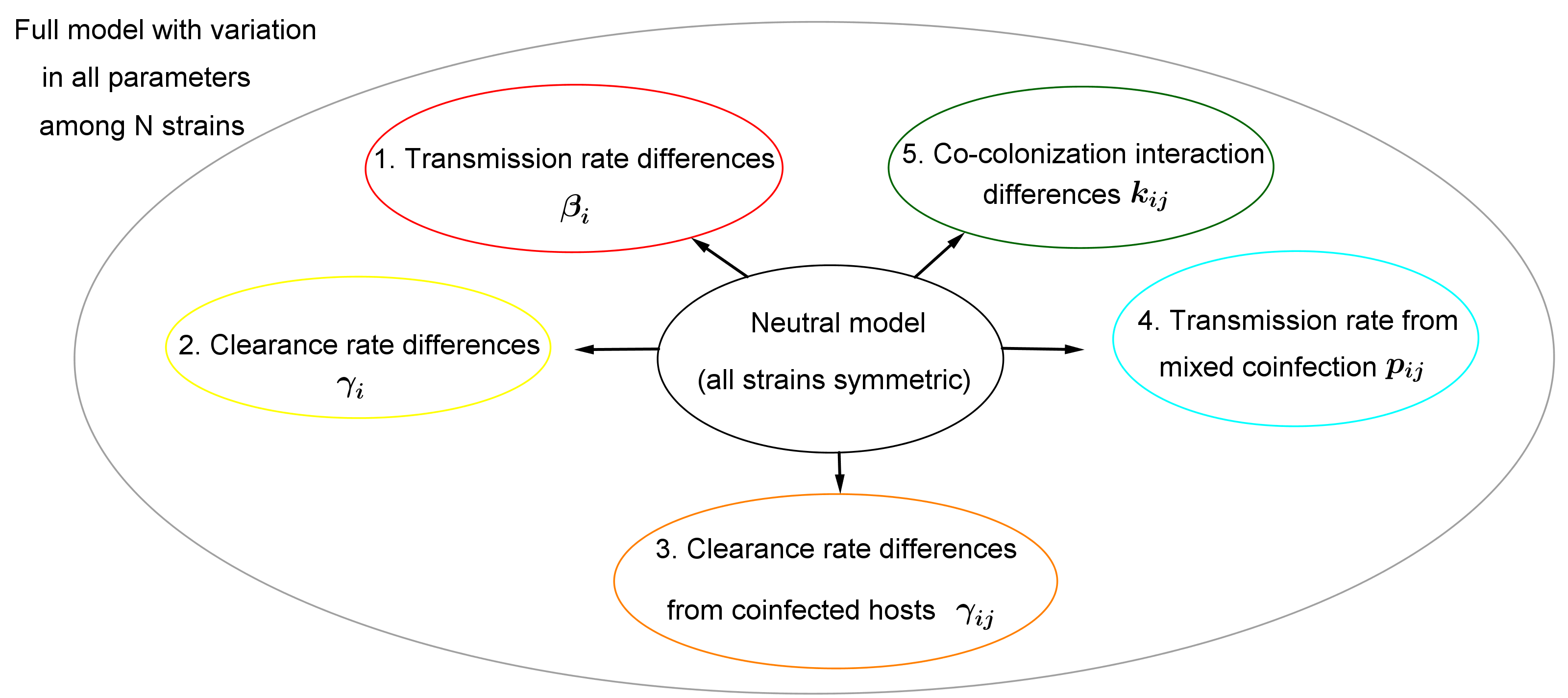}
	\caption{\textbf{Schematic description of the spirit of our study.} We study the full $N$-strain SIS model with coinfection like in \cite{Madec2, Gjini1}, but here include variation in several parameters among strains, besides co-colonization interactions. For this, we consider the neutral model as the organizing center of the dynamics, and the slow-fast form for each case of trait variation. Finally, we combine all cases of singular perturbation in each parameter to obtain the general system. Our result is the dynamics in the slow manifold, which corresponds to a replicator system for $N$ strain frequencies, governed by the pairwise invasion fitness matrix.}\label{scheme}
\end{figure}

The paper is organized as follows. Section \ref{sec2} outlines the general systems studied in this paper with corresponding quasi-neutral and neutral models. Then it introduces Tikhonov's theorem and the expansion theorem used to approximate the target model. Section \ref{section3} presents the main framework used to decompose the dynamics into fast and slow components, accompanied with lemmas and concrete steps. In this section, we state the main result: the replicator system for strain frequencies, whose coefficients' matrix is defined by pairwise invasion fitnesses. Section \ref{sec3} is devoted to the explicit computations for perturbations in each trait, and ends with the proof for the error estimate between the original system and the slow-fast approximation. In Section \ref{sec5} we provide illustration by numerical simulations about the different regimes of system behaviour, including coexistence, competitive exclusion and more complex dynamics. This helps to contextualize the competitive outcomes between strains as a function of parameters. Finally, in Section \ref{sec6} we close with conclusions and a discussion.

\section{System, methods and results} \label{sec2}
This initial section aims to provide a general description of the dynamics followed by an outline of the analytical framework applied. We first introduce the general structure, then subsequently present explicitly the steps of our approach, consisting in the quasi-neutral model, neutral model and slow-fast model. We then present the Tikhonov's theorem, which is the key tool we use to approximate the singular perturbation dynamics efficiently. The important lemmas and main results are also stated in this section.
\subsection{The general SIS coinfection model with $N$ strains and some initial analysis} \label{sec2.1}
The dynamics studied in this article groups the pathogen types in $N$ subsets, indexed by $i$, $1 \leq i \leq N$. With a set of ordinary differential equations, we then track the proportion of hosts in $1+N+N^2$ compartments: susceptible: $S$, hosts colonized by strain-$i$: $I_i$, hosts co-colonized by strain-$i$ then strain-$j$: $I_{ij}$. Notice that we include also same strain coinfection, as argued in \cite{Madec2}. We formulate the general model based on the same structure as that in \cite{Madec2} but here allow for strains to vary in their transmission rates $\beta_i$, clearance rates of single infection $\gamma_i$ (or duration of carriage $1/\gamma_i$), clearance rates from mixed co-colonization $\gamma_{ij}$, within-host competition reflected in relative transmissibilities from mixed coinfected hosts ($p_{ij}^{i}$ and $p_{ji}^i$), as well as co-colonization vulnerabilities $k_{ij}$, already studied in \cite{Madec2}.
\begin{equation} \label{2.1}
\left\{
\begin{aligned}
&\dfrac{dS}{dt} &=& r(1-S) + \sum_{i=1}^{N}\gamma_i I_i + \sum_{i,j=1}^{N}\gamma_{ij}I_{ij} - S\sum_{i=1}^{N}\beta_iJ_i,\\
&\dfrac{dI_i}{dt} &=& \beta_iJ_i S - (r + \gamma_i) I_i -  I_i\sum\limits_{j=1}^{N}k_{ij}\beta_jJ_j, \quad &1 \leq i \leq N,\\
&\dfrac{dI_{ij}}{dt} &=& k_{ij}I_i \beta_jJ_j - (r + \gamma_{ij})I_{ij}, \quad &1 \leq i,j \leq N
\end{aligned}
\right.
\end{equation}
where $J_i$ is proportion of all hosts transmitting strain $i$, including singly- and co-colonized hosts and has the explicit formula
\begin{equation} \label{2.2}
J_i = I_i + \sum_{j=1}^{N}\left(p_{ij}^iI_{ij} + p_{ji}^iI_{ji}\right).
\end{equation}
Note that $\beta_iJ_i$ is the force of infection of strain $i$, for all $i$. All mixed coinfection hosts, harboring strain $i$ (and $j$), in any order, whether acquired first or second, can transmit strain $i$ and the two probabilities of transmission are denoted by $p_{ij}^i$ and $p_{ji}^i$. The corresponding probabilities to transmit the other strain for such hosts, is simply $1- p_{ij}^i$ and $1-p_{ji}^i$ respectively. Thus we allow for variation between strains in both transmissibility from mixed coinfection, and in the benefit gained within-host for transmission when landed there first (a precedence effect). In \eqref{2.1}, for $1 \leq i,j \leq N$, parameters are summarized in Table \ref{tab:table1}.
\begin{table}[htb!]
	%   \begin{landscape}
	\caption{\textbf{Conventions and notations of parameters defining strains in our model, and host turnover. Under strain similarity assumptions, we can write each trait using a common reference for all strains, and express the variation as a deviation from neutrality, with $\epsilon$ a small number between 0 and 1.}}
	\label{tab:table1}
	\begin{center}
		\scalebox{1}{
			\begin{tabular}{ l p{2cm} p{9.5cm} p{3cm} } 
				\hline
				& \textbf{Parameter} & \textbf{Interpretation} & Strain similarity\\
				\hline
			%	\textbf{Original system} &  \\
				1. & $\beta_i $ & Strain-specific transmission rates & $\beta_i = \beta\left(1 + \epsilon b_i\right)$\\
				
				2. & $\gamma_{i} $ & Strain-specific clearance rates of single colonization & $\gamma_{i} = \gamma\left(1 + \epsilon\nu_i\right)$ \\ 
				
				3. & $\gamma_{ij}$ & Clearance rates of  co-colonization with $i$ and $j$ & $\gamma_{ij} = \gamma\left(1 + \epsilon u_{ij}\right)$\\ 
				
				4. & $p_{ij}^s $ & Transmission probability of strain $s \in \{i,j\}$ from a host co-colonized by strain-$i$ then strain-$j$, $\left(p_{ij}^i + p_{ij}^j = 1\right)$ & 
				$p_{ij}^s = \dfrac{1}{2} + \epsilon\omega_{ij}^s$\\
				
				5. & $k_{ij}$ & Relative factor of altered susceptibility to co-colonization by strain $j$ when a host is already colonized by strain $i$ & $k_{ij} = k + \epsilon\alpha_{ij}$\\

				\hline
				& $r$ & Susceptible recruitment rate (equal to natural mortality rate) \\
				& $R_0$ & Basic reproduction number & $R_0 = \frac{\beta}{\gamma + r}$\\
				\hline
		\end{tabular}
	}
	\end{center}
\end{table}
Summing up all the equations of \eqref{2.1} on both sides yields the equation for total mass
\begin{equation} \label{2.3}
\dfrac{d}{dt} \left(S + \sum_{i=1}^{N}I_i + \sum_{i,j=1}^{N}I_{ij}\right) = r(1-S) - r\left(\sum_{i=1}^{N}I_i + \sum_{i,j=1}^{N}I_{ij}\right),
\end{equation}
which leads to $S + \sum_{i=1}^{N}I_i + \sum_{i,j=1}^{N}I_{ij} = 1 - e^{-rt}$. Hence, $S + \sum_{i=1}^{N}I_i + \sum_{i,j=1}^{N}I_{ij}$ tends to 1 as $t \to \infty$.\\ \\
We want to study a system whose host population is invariant. Such an expectation leads to the assumption that, \eqref{2.1} has the same recruitment rate of susceptibility host and mortality rate of strains. It is plausible to from now on assume that the total population size is constant and rescaled to unit. We also take the system \eqref{2.1} as given the initial conditions $S(0) + \sum_{i=1}^{N}I_i(0) + \sum_{i,j=1}^{N}I_{ij}(0) = 1$, which implies that the total population size is always one for any time. Thus our compartmental variables can be taken to reflect proportions of host in different epidemiological states.

%%%%%%%%%%%%%%%%%%%%%%%%%%%%%%%%%%%%%%%%%%%%%%%%%%%%%%%%%%%%%%%%%%%%%%%%%%%%%%%%%%%%%
%%%%%%%%%%%%%%%%%%%%%%%%%%%%%%%%%%%%%%%%%%%%%%%%%%%%%%%%%%%%%%%%%%%%%%%%%%%%%%%%%%%%%
\subsection{Quasi-neutral system and new variables}\label{sec2.2}
A straightforward understanding of \eqref{2.1} is not possible due to its complexity, high-dimensional parameter space and number of equations. However, for indistinguishable strains, i.e. if all the parameters do not depend on the strain $i$, we obtain the so-called \textit{neutral system} which is analytically tractable (see \cite{Gjini,Gjini1,Madec2}). In this text, we make a \textit{quasi neutral} assumption by assuming that the parameters are nearly equal, because the strains are similar. Without loss of generality we can take the same epsilon in all parameters with the perturbations written in the form presented in table \ref{tab:table1}. 
%\begin{equation*} 
%\begin{aligned}
%&\beta_i = \beta\left(1 + \epsilon b_i\right); \qquad \gamma_i = \gamma\left(1 + \epsilon\nu_i\right); \qquad \gamma_{ij} = \gamma\left(1 + \epsilon  u_{ij}\right);\\
%&p_{ij}^s = \dfrac{1}{2} + \epsilon \omega_{ij}^s \quad s\in\{i,j\} \quad  \left(\omega_{ij}^i + \omega_{ji}^j = 0\right); \qquad k_{ij} = k + \epsilon\alpha_{ij}.
%\end{aligned}
%\end{equation*}
For the sake of simplicity, we denote the inverse duration of a carriage episode by strain $i$ with $m_i = r + \gamma_i$, of a co-carriage episode by strains $i$ and $j$ with $m_{ij} = r + \gamma_{ij}$ and the corresponding inverse duration of carriage if all strains were equivalent with $m = r + \gamma$.

%%%%%%%%%%%%%%%%%%Original place of this generatilies%%%%%%%%%%%%%%%%%%%%%%%
%%%%%%%%%%%%%%%%%%%%%%%%%%%%%%%%%%%%%%%%%%%%%%%%%%%%%%%%%%%%%%%%%%%%%%%%%%%%%%%%%%%%%%%
To work on the neutral system, it's useful to denote some new state variables, including  the total `mass' of singly-infected hosts $I$, the total `mass' of doubly-infected hosts $D$, and the total `mass' of infected hosts $T=I+D$. According to these definitions of $T,\;I,\;D$, we have the formulae:
\begin{equation} \label{2.5}
 I = \sum_{i=1}^N I_i, \quad D=\sum_{i,j=1}^N I_{ij},
\quad T = I+D.
\end{equation}
It can be easily deduced from \eqref{2.5} together with $\omega_{ij}^i + \omega_{ji}^j = 0$ that $\sum_{i=1}^N J_i=T$.
Thanks to these new variables, the original system \eqref{2.1} can be rewritten into the extensive new form
\begin{equation} \label{2.8}
\left\{
\begin{aligned}
&\dfrac{dS}{dt} &=& r(1-S) + \gamma T + \epsilon \gamma \left(\sum_{i=1}^N  \nu_i I_i + \sum_{i,j=1}^N  u_{ij}I_{ij}\right) - \beta ST - \epsilon\beta S \sum_{i=1}^{N}b_i J_i \\
&\dfrac{dT}{dt} &=& \beta ST - mT + \epsilon \beta S \sum_{i=1}^{N}b_i J_i  -\epsilon\gamma\left(\sum_{i=1}^{N}\nu_iI_i + \sum_{i,j=1}^N  u_{ij}I_{ij}\right)\\
&\dfrac{dI_i}{dt} &=& \beta J_i S+\epsilon\beta b_i J_iS -(m + \epsilon\gamma\nu_i)I_i - \beta I_i \sum_{j=1}^{N}\left(k + \epsilon \alpha_{ij}\right)\left(1 + \epsilon b_j\right)J_j \\
&\dfrac{dJ_i}{dt} &=& \beta(1+\epsilon b_i)J_i S - \beta I_i\sum_{j = 1}^N(k + \epsilon \alpha_{ij})(1+\epsilon b_j)J_j   
-\epsilon\gamma\left[\nu_iI_i + \sum_{j = 1}^N\left((\frac12+\epsilon\omega^i_{ij}) u_{ij}I_{ij} + (\frac12+\epsilon\omega^i_{ji}) u_{ji}I_{ji}\right)\right]\\
&&&\quad- mJ_i +\beta\sum_{j=1}^{N}\left((\frac12+\epsilon\omega_{ij}^i)(k + \epsilon \alpha_{ij})(1+\epsilon b_j)I_i J_j + (\frac12+\epsilon\omega_{ji}^i)
\left(k + \epsilon \alpha_{ji}\right)(1+\epsilon b_i)I_j J_i\right)\\
&\dfrac{dI}{dt} &=& \beta T S + \epsilon\beta S\sum_{i=1}^{N}b_i J_i - mI - \epsilon\gamma\sum_{i=1}^N \nu_iI_i 
- \beta\sum_{i=1}^N I_i\left(\sum_{j = 1}^N (k + \epsilon \alpha_{ij})(1+\epsilon b_j)J_j\right)\\
&\dfrac{dI_{ij}}{dt} &=& \beta\left(k + \epsilon \alpha_{ij}\right)(1 + \epsilon b_j)I_i J_j - (m + \epsilon\gamma u_{ij})I_{ij}.
\end{aligned}
\right.
\end{equation}
This system has the generic  form $\dfrac{dX}{dt} = \tilde{F}(X,\epsilon)$ where $X = (X_1,X_2,\dots,X_{\tilde{n}} ) \in \mathbb{R}^{\tilde{n}}$ (for some integer $\tilde{n}$) and is equivalent to $\dfrac{dX}{dt} = F(X) + O(\epsilon)$ after some algebraic transformations. In ours case, the part $\dfrac{dX}{dt} = F(X)$ is known as the \textit{neutral system}, consistently stays unaltered and be investigated in the subsection \ref{sec2.3}. It is important to note that this neutral system is structurally unstable. Then, the part $O\left(\epsilon\right)$ is a singular perturbation of the neutral system. To treat such an emergence by Tikhonov's theorem, it's essential to rewrite $\dfrac{dX}{dt} = F(X) + O(\epsilon)$ into an equivalent slow-fast form
\begin{equation} \label{2.4}
\left\{
\begin{aligned}
&\dfrac{dx}{dt} = &\epsilon\left(f(x,y) + O(\epsilon)\right)\\
&\dfrac{dy}{dt} = &g(x,y) + O(\epsilon)
\end{aligned}
\right.
\end{equation}
where $y \in \mathbb{R}^{n_y}$ is the fast variable and $x \in \mathbb{R}^{n_x}$ is the slow variable (with $n_x+n_y=\widetilde{n}$). In general,  the finding of this slow-fast reformulation is  strongly dependent on the specific system. Here, it is achieved thanks to the ansatz \eqref{2.24} which is  yielded from the study of the neutral system.\\
Hence, we start to study the important {\it neutral system} which is obtained for $\epsilon = 0$ in \eqref{2.8}. This study yields  the definition of the appropriate  slow and fast variables  $(v_i,z_i)$. These variables together with the ansatz \eqref{2.24} are the key for the slow-fast study of the next section. \\ 

%%%%%%%%%%%%%%%%%%%%%%%%%%%%%%%%%%%%%%%%%%%%%%%%%%%%%%%%%%%%%%%
\subsection{Neutral system, $\epsilon = 0$}\label{sec2.3}
%%%%%%%%%%%%%%%%%%%%%%%%%%%%%%%%%%%%%%%%%%%%%%%%%%%%%%%%%%%%%%%
Taking $\epsilon = 0$ in \eqref{2.8} leads to the so-called {\it Neutral System\footnote{The name {\it neutral system} comes from the fact that if $\epsilon=0$ then the parameters do not depend on the strains as in the neutral theory, and the model describes indistinguishable strains. 
}} for $S,T,I,I_i, J_i, I_{ij}$ which reads after some simplifications:
\begin{equation} \label{2.9}
\left\{
\begin{aligned}
&\dfrac{dS}{dt} = r(1-S) + \gamma T -  S\beta T  \\
&\dfrac{dT}{dt} = S\beta T - m T \\
&\dfrac{dI}{dt} = \beta T S - (m+k\beta T) I  \\
&\dfrac{dI_i}{dt} = \beta J_i S -m I_i -  k I_i\beta T, \quad &&1 \leq i \leq N \\
&\dfrac{dJ_i}{dt} = (\beta  S-m )J_i  + \frac{1}{2}\beta kI J_i- \dfrac{1}{2}\beta kI_i T , &&1 \leq i \leq N \\
&\dfrac{dI_{ij}}{dt} = k\beta I_i J_j - m  I_{ij}, &&1 \leq i,j \leq N.
\end{aligned}
\right.
\end{equation}
Such a triangular structure of this system enables to successively consider the subsystems for $(S,T)$, $I$, $(I_i,J_i)$ and $I_{ij}$.\\ \\
$\large\bullet$ \textit{Firstly}, we consider the neutral system for $S,T$ as following
\begin{equation} \label{2.10}
\left\{
\begin{aligned}
&\dfrac{dS}{dt} &=& m(1-S) - \beta ST\\
&\dfrac{dT}{dt} &=& -mT + \beta ST
\end{aligned}
\right.
\end{equation}
This system is a classical. As in \cite{Ma}, we define the basic reproduction number as $R_0 = \dfrac{\beta}{m}$.
If $R_0>1$ then it admits a positive steady state  $(S^*,T^*)$ where $S^* = \dfrac{1}{R_0}$ and $T^*= 1-S^*$.\\ \\
We now recall a crucial proposition, which follows the definition of $S^*$ and $T^*$.
\begin{prop} \label{thm1} Assume that $S(0)>0$ and $T(0)>0$.
	If $R_0 \leq 1$ then the solution $S,T$ of system \eqref{2.10} tends to $(1,0)$. Otherwise, it tends to $(S^*,T^*)$ asymptotically.
\end{prop}
The proof for this Proposition can be found in \cite{Murray}.\\ \\
$\large\bullet$ \textit{Secondly}, we prove that $I(t) \to I^*:= \dfrac{mT^*}{m + \beta k T^*}$.\\ \\
Indeed, substitute $(S,T)$ by $\left(S^* + (S - S^*), T^* + (T - T^*)\right)$ into the equation of $I$ in \eqref{2.10} then make some manipulations to obtain
\begin{equation} \label{2.11}
\dfrac{dI}{dt} = m T^*-(m+\beta k T^*)I + \left[\beta S^*\left(T - T^*\right) + \beta T^*\left(S - S^*\right) + \beta\left(T - T^*\right)\left(S - S^*\right)\right].
\end{equation}
Consider the equation
\begin{equation} \label{2.12}
\dfrac{d\tilde{I}}{dt} = m T^*-(m+\beta k T^*)\tilde{I}
\end{equation}
which has the explicit solution $\tilde{I}(t) = \dfrac{mT^*}{m+\beta k T^*}\left(1- \dfrac{m+\beta k T^*}{mT^*}I(0)\exp\left(-(m+\beta k T^*)t\right) \right)$.
We simultaneously have the equation for $I - \tilde{I}$ as follows
\begin{equation} \label{2.13}
\dfrac{d}{dt}\left(I - \tilde{I}\right) = -(m+\beta k T^*)\left(I - \tilde{I}\right) + \left[\beta S^*\left(T - T^*\right) + \beta T^*\left(S - S^*\right) + \beta\left(T - T^*\right)\left(S - S^*\right)\right].
\end{equation}
Set $f\left(t\right) = \beta S^*\left(T - T^*\right) + \beta T^*\left(S - S^*\right) + \beta\left(T - T^*\right)\left(S - S^*\right)$ then $f\left(t\right) \to 0$ asymptotically when $t \to \infty$, by Proposition \ref{thm1}. It's easy to see $\left(I - \tilde{I}\right) = \exp\left(-(m+\beta k T^*)t\right)\left(\int_{0}^{t}\exp\left((m+\beta k T^*)s\right)f(s)ds + C\right)$, with $C$ is some suitable constant. Hence, $I(t) - \tilde{I}(t) \to 0$ when $t \to \infty$ then leads to $I(t) \to I^*$ as $t \to \infty$.\\ \\
For  later reference, we also write their equilibrium values in the neutral system
\begin{equation} \label{2.6}
\begin{aligned}
S^* = \dfrac{m}{\beta}, \qquad T^* = 1 - \dfrac{m}{\beta}, \qquad I^* = \dfrac{mT^*}{m + \beta k T^*}, \qquad D^* = T^*-I^*=\dfrac{\beta k {T^*}^2}{m + \beta k T^*}.
\end{aligned}
\end{equation}
$\large\bullet$\textit{Thirdly}, from \eqref{2.9}, we also have the neutral model for $I_i,J_i$ for all $1 \leq i \leq N$. This  is the  very important part which gives crucial insight for  $0<\epsilon \ll 1$ in the next section. For now, $\epsilon = 0$ and substitute $(S,T,I)$ by the limit $(S^*,T^*,I^*)$, we obtain the (degenerate) linear system
\begin{equation} \label{2.14}
\begin{aligned}
\dfrac{d}{dt}\begin{pmatrix}
I_i \\ J_i
\end{pmatrix} =& \begin{pmatrix}
-(m+\beta k T^*) & m \\ -\dfrac{\beta k T^*}{2} & \dfrac{\beta k I^*}{2}
\end{pmatrix} \begin{pmatrix}
I_i \\J_i 
\end{pmatrix}
\end{aligned}.
\end{equation}
Set $A = \begin{pmatrix}
-(m+\beta k T^*) & m \\ -\dfrac{\beta k T^*}{2} & \dfrac{\beta k I^*}{2}
\end{pmatrix} $, $D^* = T^* - I^*$  and
\begin{equation} \label{2.15}
P = \begin{pmatrix}
 2T^*& I^* \\
 D^*& T^*
\end{pmatrix}, \qquad P^{-1} = \dfrac{1}{|P|}\begin{pmatrix}
T^* & -I^* \\ -D* & 2T^*
\end{pmatrix}\text{ and for $i=1,\cdots,N$ }
\begin{pmatrix}
v_i \\ z_i
\end{pmatrix} = P^{-1} \begin{pmatrix}
I_i \\ J_i
\end{pmatrix}
\end{equation}
We have $A=P \begin{pmatrix}0&0\\0&-\xi \end{pmatrix}P^{-1}$  where $\xi = m + \beta k T^* - \dfrac{1}{2}\beta k I^*>m + \dfrac{1}{2}\beta k (T^* - I^*) > 0$ and $|P| =  2{T^*}^2-I^*D^*  >0$.
\\ \\ 
From  \eqref{2.14} and \eqref{2.15}, we  infer an equation for $\begin{pmatrix}
v_i \\ z_i
\end{pmatrix}$ for each $1 \leq i \leq N$:
\begin{equation} \label{2.16}
\left\{\begin{aligned}
&\dfrac{dv_i}{dt} &=& -\xi v_i\\
&\dfrac{dz_i}{dt} &=& 0. 
\end{aligned} \right.
\end{equation}
This step of changing to $(v_i,z_i)$ plays an important role. Since under these new variables, we can rewrite into the slow-fast form. It allows us to apply the Tikhonov's Theorem introduced in the next subsection.\\
Let us remark that $z_i$ is exactly frequency of strain $i$ in the total of infected, see the proof in \cite{Madec2}.\\ \\
$\large\bullet$ \textit{Fourthly},  the $N^2$ last equations for $I_{ij}$ in \eqref{2.9} yields $1 \leq i \leq N$
\begin{equation} \label{2.17}
\dfrac{dI_{ij}}{dt} = \beta k I_i J_j - m I_{ij}.
\end{equation}
Whose dynamics is trivial once $I_i$ and $J_i$ are known. Indeed, assume that for each $i$, there exists $\left(\tilde{I}_i,\tilde{J}_i\right)$ such that $I_i(t) - \tilde{I}_i(t) = O\left(\epsilon\right)$ and $J_i(t) - \tilde{J}_i(t) = O\left(\epsilon\right)$, then we can rewrite \eqref{2.17} into
\begin{equation} \label{2.18}
\dfrac{dI_{ij}}{dt} = - m I_{ij} + \beta k \tilde{I}_i \tilde{J}_j + \beta k \left[\left(I_i - \tilde{I}_i\right)\tilde{J}_j + \left(J_j - \tilde{J}_j\right)\tilde{I}_i + \left(I_i - \tilde{I}_i\right)\left(J_j - \tilde{J}_j\right)\right].
\end{equation}
Consider the equation
\begin{equation} \label{2.19}
\dfrac{d\tilde{I}_{ij}}{dt} = - m \tilde{I}_{ij} + \beta k \tilde{I}_i \tilde{J}_j
\end{equation}
then we can obtain the differential equation for $I_{ij} - \tilde{I}_{ij}$
\begin{equation} \label{2.20}
\dfrac{d}{dt}\left(I_{ij} - \tilde{I}_{ij}\right) = - m \left(I_{ij} - \tilde{I}_{ij}\right) + \beta k \left[\left(I_i - \tilde{I}_i\right)\tilde{J}_j + \left(J_j - \tilde{J}_j\right)\bar{I}_i + \left(I_i - \tilde{I}_i\right)\left(J_j - \tilde{J}_j\right)\right].
\end{equation}
By our assumption on $\tilde{I}_i,\tilde{J}_j$ and use the same arguments for $I(t) \to I^*$, we deduce that $I_{ij}(t) - \tilde{I}_{ij}(t) = O\left(\epsilon\right)$ on each bounded interval of time.
%%%%%%%%%%%%%%%%%%%%%%%%%%%%%%%%%%%%%%%%%%%%%%%%%%%%%%%%%%%%%%%
\subsection{Tikhonov's Theorem and derivation of the non-neutral dynamics}\label{sec2.4}
%%%%%%%%%%%%%%%%%%%%%%%%%%%%%%%%%%%%%%%%%%%%%%%%%%%%%%%%%%%%%%%
Using the above idea, we transform the problem into an equivalent slow-fast form which is analyzed through singular perturbations method. According to previous arguments, our slow-fast form includes variables $\left(X,Y,\mathbf{L},\mathbf{v},\mathbf{z}\right)$. Using \eqref{2.15}, we define  $\begin{pmatrix}
I_i \\ J_i
\end{pmatrix} = P \begin{pmatrix}
v_i \\ z_i
\end{pmatrix}.$
Proceeding like in \eqref{2.16}, we obtain for $\epsilon>0$:
\begin{equation} \label{2.27}
\left\{
\begin{aligned}
&\dfrac{dv_i}{dt} &=& -\xi v_i + O(\epsilon)\\
&\dfrac{dz_i}{dt} &=& O(\epsilon).
\end{aligned} 
\right.
\end{equation}
By setting $\tau = \epsilon t$, \eqref{2.28} can be read as the slow time scale:
\begin{equation} \label{2.28}
\left\{
\begin{aligned}
\epsilon&\dfrac{dv_i}{d\tau} &=& -\xi v_i + O\left(\epsilon\right)\\
&\dfrac{dz_i}{d\tau} &=& O(1).
\end{aligned}
\right.
\end{equation}
%%%%%%%%%%%%%%%%%%%%%%%%%%%%%%%%%%%%%%%%%%%%%%%%%%%%%%%%%%%%%%%%%%%%%%%%%%%%%%%%%%%%%%%%%%%%%%%%%%%%%%%%%%%%%%%%%%%%%%%%%%%
%%%%%%%%%%%%%%%%%%%%%%%%%%%%%%%%%%%%%%%%%%%%%%%%%%%%%%%%%%%%%%%%%%%%%%%%%%%%%%%%%%%%%%%%%%%%%%%%%%%%%%%%%%%%%%%%%%%%%%%%%%%
%%%%%%%%%%%%%%%%%%%%%%%%%%%%%%%       New version of Madec   26/01/2021     %%%%%%%%%%%%%%%%%%%%%%%%%%%%%%%%%%%%%%%%%%%%%%%
%%%%%%%%%%%%%%%%%%%%%%%%%%%%%%%%%%%%%%%%%%%%%%%%%%%%%%%%%%%%%%%%%%%%%%%%%%%%%%%%%%%%%%%%%%%%%%%%%%%%%%%%%%%%%%%%%%%%%%%%%%%
%%%%%%%%%%%%%%%%%%%%%%%%%%%%%%%%%%%%%%%%%%%%%%%%%%%%%%%%%%%%%%%%%%%%%%%%%%%%%%%%%%%%%%%%%%%%%%%%%%%%%%%%%%%%%%%%%%%%%%%%%%%
We need to compute explicitly the perturbation $O(1)$ in \eqref{2.28}. This computation is quite complex especially when involving perturbation in each parameters, so it's worthwhile  dividing this progress into five sub-cases wherein only one perturbation at a time occurs.\\ \\
%before treating the full and general one.\\ \\
After that, we will treat the slow-fast form by the Tikhonov's theorem, that is presented as follows.
\begin{thm}[Tikhonov, 1952, see \cite{Tik}] \label{thm2}
	Consider the initial value problem
	\begin{equation} \label{2.21}
	\left\{
	\begin{aligned}
	&\dfrac{dx}{d\tau} &=& f(x,y,\tau) + \epsilon \dots, \quad &x(0) = x_0, \quad &x \in D \subset \mathbb{R}^n,\\
	\epsilon&\dfrac{dy}{d\tau} &=& g(x,y,\tau) + \epsilon \dots, \quad &y(0) = y_0,\quad &y \in G \subset \mathbb{R}^n. 
	\end{aligned}
	\right.
	\end{equation}
	For $f$ and $g$, we take sufficiently smooth vector functions in $x$, $y$ and $t$; the dots represent (smooth) higher-order terms in $\epsilon$.
	\begin{enumerate}
		\item [a.] We assume that a unique solution of the initial value problem exists and suppose this holds also for the reduced problem
		\begin{equation} \label{2.22}
		\left\{
		\begin{aligned}
		&\dfrac{dx}{d\tau} &=& f(x,y,\tau), \quad &x(0) = x_0,\\
		&0 &=& g(x,y,\tau),
		\end{aligned}
		\right.
		\end{equation}
		with solution $\bar{x}(\tau)$, $\bar{y}(\tau)$.
		\item[b.] Suppose that $0 = g(x,y,\tau)$ is solved by $\bar{y} = \phi(x,\tau)$, where $\phi(x,\tau)$ is a continuous function and an isolated root, i.e. there exists a neighbor of $\phi(x,\tau)$ such that there is no other solution for $0 = g(x,y,\tau)$ in this vicinity. Also, suppose that $\bar{y} = \phi(x,t)$ is an asymptotically stable solution		\footnote{Recall that the solution $\bar{y} = \phi(x,\tau)$ is asymptotically stable if for each $\tau_0 > 0$, a $\delta(\tau_0)$ can be found such that: $\|y_0 - \phi(x,\tau_0)\| \leq \delta(\tau_0)$ yields $\lim\limits_{\tau \to \infty}\|y(\tau;\tau_0,x_0) - \phi(x,\tau)\| = 0$.} of the equation $\dfrac{dy}{dt} = g(x,y,\tau)$, where $\tau = \epsilon t$, that is uniform in the parameters $x \in D$ and $t \in \mathbb{R}^+$.\\ \\ 

		\item [c.] $y(0)$ is contained in an interior subset of the domain of attraction of $\bar{y} = \phi(x,\tau)$ in the case of the parameter values $x = x(0)$, $\tau = 0$.
	\end{enumerate}
	Then, we have
	\begin{equation} \label{2.23}
	\begin{aligned}
	&\lim\limits_{\epsilon \to 0} x_\epsilon (\tau) &=& \bar{x}(\tau), \quad &0 \leq \tau \leq T,\\
	&\lim\limits_{\epsilon \to 0} y_\epsilon (\tau) &=& \bar{y}(\tau), \quad &0 <\tau_0 \leq \tau \leq T,
	\end{aligned}
	\end{equation}
	with $\tau_0$ and $T$ are constants independent of $\epsilon$.
\end{thm}
Beside, it needs to use another result that allows us to approximate the original system by the slow-fast form. The following error estimate gives a more precise description of these limits.
 (theorem 9.1, \cite{Ver} adapted here for  the simple case $m=0$).
	\begin{thm} \label{thm5}
	[see \cite{Ver}] Consider the initial value problem
	\begin{equation} \label{2.24}
	\dfrac{dx}{dt} = f_0(t,x) +\epsilon R(t,x,\epsilon)
	\end{equation}
	with $x(t_0) = \eta$ and $|t - t_0| \leq h$, $x \in D \subset \mathbb{R}^n$, $0 \leq \epsilon \leq \epsilon_0$. Assume that in this domain we have
	\begin{enumerate}
		\item [a.] $f(t,x)$  continuous in $t$ and $x$, $2$ times continuously differentiable in $x$;
		\item [b.] $R(t,x,\epsilon)$ continuous in $t,x$ and $\epsilon$, Lipschitz-continuous in $x$.
	\end{enumerate}
	Let $x_0(t)$ be the solution of 
		\begin{equation} \label{2.25}
	\dfrac{dx}{dt} = f_0(t,x) 
	\end{equation}
	with $x_0(t_0) = \eta$
Let $T>0$ and assume that both $x$ and $x_0$ are defined on $[0,T]$ for any $\epsilon \in (0,\epsilon_0)$. There exist $C>0$ (depending on $T$) such that for any $\epsilon \in(0,\epsilon_0)$, and $t\in (0,T)$, we have the estimate
	\begin{equation} \label{2.26}
	\left\|x(t) -x_0(t)  \right\| \leq C \epsilon
	\end{equation}
\end{thm}

\section{Integrating many perturbations in the slow-fast approximation}\label{section3}

\subsection{Steps for application of Tikhonov's theorem in our system}
Next we develop a lemma showing allowing to linearly combine all the relevant simple cases directly into the slow equation. For this purpose, we use the following notations in system \eqref{2.8}.
\begin{equation}
\begin{aligned}
&\beta_i = \beta\left(1 + \chi_1\epsilon b_i\right);\qquad \gamma_i = \gamma\left(1 + \chi_2\epsilon\nu_i\right); \qquad \gamma_{ij} = \gamma\left(1 + \chi_3\epsilon  u_{ij}\right);\\ 
&p_{ij}^s = \dfrac{1}{2} + \chi_4\epsilon \omega_{ij}^s \quad s\in\{i,j\} \quad \left(\omega_{ij}^i + \omega_{ij}^j = 0\right); \qquad k_{ij} = k + \chi_5\epsilon\alpha_{ij};
\end{aligned}
\end{equation}
where $\chi_d \in \{0,1\}$ for $d=1,2,3,4,5$.\\ \\
Any combination of trait variation among strains, can be captured via $\mathcal{A}$ where $\mathcal{A}$ is a subset of $\{1,2,3,4,5\}$ denoting the absence/presence of perturbations in that parameter among strains: for some fixed initial values given, let $C_\mathcal{A}$ be the system \eqref{2.8} with $\chi_d = 1$ if $d \in \mathcal{A}$ and $\chi_d = 0$ if $d \notin \mathcal{A}$. For simplicity, we note also $C_{\{d\}}$ by $C_d$ for $d \in \{1,2,3,4,5\}$.
\begin{rmk}
If $\mathcal{A}=\emptyset$ then there is no perturbation and the system $C_{\emptyset}$ is exactly the {\it neutral} model \eqref{2.9}.
If $\mathcal{A}=\left\{5\right\}$ then $C_5$ is the system with perturbation on the co-colonization interaction parameters $k_{ij}$ only, that has been  studied in \cite{Gjini,Gjini1,Madec2}.
\end{rmk}
In order to capture all the perturbations of order 1 in  the equation of the $z_i$ we need these additional changes of variables:
\begin{equation}\label{2.30}
S(t) = S^* - \epsilon X (t)+ O(\epsilon^2); \quad T(t)=T^* + \epsilon X(t) + O(\epsilon^2); \quad I(t) = I^* +\epsilon Y(t) + O(\epsilon^2).
\end{equation}
where $S^*$, $T^*$ and $I^*$ are defined in \eqref{2.6}, and for $i=1 ,\cdots, N$:
\begin{equation}\label{2.31}L_i (t) = \dfrac{1}{2}\sum\limits_{j=1}^{N}\left( u_{ij}I_{ij}(t) +  u_{ji}I_{ji}(t)\right)
.%\text{  and }K_i(t) = \sum\limits_{j=1}^{N}\omega_{ij}I_{ij}(t).
\end{equation}
With these notations, $C_{\mathcal{A}}$ reads
\begin{equation} \label{2.32}
\left\{
\begin{aligned}
&\dfrac{dX}{dt} &=& -\beta T^*X +  \chi_1\beta S^*\sum_{i=1}^N b_i J_i - \chi_2\gamma\sum_{i=1}^{N}\nu_i I_i -  \chi_3\gamma\sum_{i=1}^{N}L_i     + O(\epsilon)\\
&\dfrac{dY}{dt} &=& \beta(S^* - T^* -k I^*)X - (m+\beta k T^*)Y + \chi_1\beta(S^*-kI^*)\sum_{i=1}^N b_i J_i  - \chi_2\gamma\sum_{i=1}^{N}\nu_i I_i  - \chi_5\beta \sum_{i,j=1}^{N}\alpha_{ij}I_i J_j + O(\epsilon)\\
&\dfrac{dL_i}{dt} &=& -m L_i + \chi_3\dfrac{1}{2}\beta\gamma k I_i \sum_{j = 1}^N u_{ij}J_j + \chi_3\dfrac{1}{2}\beta \gamma k J_i \sum_{j=1}^{N} u_{ji}I_j + O\left(\epsilon\right)
\end{aligned}
\right.
\end{equation}
together with (we omit terms of $O\left(\epsilon^2\right)$)
\begin{equation} \label{2.33}
\begin{aligned}
\dfrac{d}{dt}\begin{pmatrix}
I_i \\ J_i 
\end{pmatrix} &=&
 A \begin{pmatrix}
I_i \\J_i
\end{pmatrix}  &- \epsilon \beta \begin{pmatrix}
k & 1 \\ 
\frac{k}{2} & 1
\end{pmatrix} \begin{pmatrix}
I_i \\ J_i
\end{pmatrix} X + \epsilon \dfrac{\beta k}{2} \begin{pmatrix}
0 & 0  \\ 
0 & 1
\end{pmatrix} \begin{pmatrix}
I_i \\ J_i
\end{pmatrix} Y + \epsilon \mathcal{M}_{\mathcal{A}} \begin{pmatrix}
I_i \\ J_i
\end{pmatrix} - \epsilon \chi_3\begin{pmatrix}
0 \\ L_i
\end{pmatrix} 
\end{aligned}
\end{equation}
where $A$ is defined in \eqref{2.14} and  $\mathcal{M}_{\mathcal{A}}$ is the matrix
\begin{equation} \label{2.34}
\begin{pmatrix}
%M11%
-\chi_1\beta k \sum\limits_{i=1}^N b_i J_i - \chi_2\gamma\nu_i - \chi_5\beta\sum\limits_{j=1}^{N}\alpha_{ij}J_j & 
%M12%
\chi_1\beta b_i S^*  \\ 
%M21%
\beta\sum\limits_{j=1}^{N}\left(\chi_4k\omega_{ij}^i -\chi_5\frac{\alpha_{ij}}{2} \right)J_j -\chi_1\beta\frac{k}{2}\sum\limits_{i=1}^N b_i J_i  - \chi_2\gamma\nu_i  &
%M22% 
\chi_1\beta b_i\left(S^* +\frac{k I^*}{2}\right)  + \beta\sum\limits_{j=1}^{N}\left(\chi_4k\omega_{ji}^i + \chi_5\frac{\alpha_{ji}}{2} \right)I_j
\end{pmatrix}
\end{equation}
In order to apply the Theorem \eqref{thm2}, we rewrite system $C_{\mathcal{A}}$ using the changes of variables detailed in \eqref{2.15}.\\
Let us note $$\mathbf{L}=(L_i)_i,\quad\mathbf{v}=(v_i)_i,\quad\mathbf{z}=(z_i)_i,$$
and  $-\xi=-(m+\beta k T^*)+ \dfrac{\beta k I^*}{2}<0$. The system $C_{\mathcal{A}}$ reads now as the slow-fast form
\begin{equation} \label{CA}
\left\{
\begin{aligned}
&\dfrac{dX}{dt} &=& -\beta T^*X +\chi_1 F_X^1\left(\mathbf{v},\mathbf{z}\right)
+\chi_2 F_X^2\left(\mathbf{v},\mathbf{z}\right)
+\chi_3 F_X^3\left(\mathbf{L}\right)
% \chi_i F_X^i\left(\mathbf{L},\mathbf{K},\mathbf{v},\mathbf{z}\right)
+ O(\epsilon)\\
&\dfrac{dY}{dt} &=& \beta(S^* - T^* -k I^*)X - (m+\beta k T^*)Y 
+\chi_1 F_Y^1\left(\mathbf{v},\mathbf{z}\right)
+\chi_2 F_Y^2\left(\mathbf{v},\mathbf{z}\right)
+\chi_5 F_Y^5\left(\mathbf{v},\mathbf{z}\right)
%+ \sum_{i\in\{1,2,4,5\}} \chi_i F_Y^{i}\left(\mathbf{K},\mathbf{v},\mathbf{z}\right)
+ O(\epsilon)\\
&\dfrac{dL_i}{dt} &=& -m L_i + \chi_3 F_{L_i} \left(\mathbf{v},\mathbf{z}\right)+ O\left(\epsilon\right)\\
&\dfrac{dv_i}{dt} &=& -\xi v_i + O(\epsilon)\\
&\dfrac{dz_i}{dt} &=& \epsilon\left(F_{z_i}(X,Y,\mathbf{L},\mathbf{v},\mathbf{z})+O(\epsilon)\right)
\end{aligned}
\right.
\end{equation}
wherein we have replaced $I_i$ and $J_i$ by $v_i$ and $z_i$ though the change of variable \eqref{2.15}, that is:
$$\begin{pmatrix}I_i\\J_i \end{pmatrix}=P\begin{pmatrix}v_i\\z_i \end{pmatrix}\text{ with } P=\begin{pmatrix}2T^*&I^*\\D^*&T^* \end{pmatrix}.$$
For $i=1,\cdots,N$, the functions $F_X^{i}$, $F_Y^{i}$ and $F_{L_i}$ are obviously deduced from the right term of \eqref{2.32} and are linear in theirs variables, $X,Y$ and $\mathbf{L}$ respectively.
The function $F_Y^{4}$ is quadratic in $(\mathbf{v},\mathbf{z})$.
Finally, $F_{z_i}$ is given by the second line of the right term of \eqref{2.34}  after the linear change of variables \eqref{2.15}:
\begin{equation}\label{eq:Fzi}
F_{z_i}\left(X,Y,\mathbf{L},\mathbf{v},\mathbf{z}\right)=\begin{pmatrix} 0 & 1\end{pmatrix}P^{-1}\left(\beta \begin{pmatrix}
-k & -1 \\ 
-\dfrac{k}{2} & -1
\end{pmatrix}  X +  \dfrac{\beta k}{2} \begin{pmatrix}
0 & 0  \\ 
0 & 1
\end{pmatrix}  Y + \mathcal{M}_{\mathcal{A}} \right)
 P \begin{pmatrix} v_i\\z_i\end{pmatrix} + \begin{pmatrix} 0 & 1\end{pmatrix}P^{-1}\chi_3\gamma\begin{pmatrix} 0 \\ L_i\end{pmatrix}.
\end{equation}
The next step is to change the time scale. Taking $\tau=\epsilon t$ in \eqref{CA} we obtain\footnote{We use the usual notation abuse. Rigorously speaking, we have to define $\widetilde{X}(\tau)=X\left(\frac{\tau}{\epsilon}\right)$ and the same for each variables. Here we remove the $\widetilde{}$ for simplicity.} 
the following system which is equivalent to \eqref{CA} but in the slow motion $\tau$.
\begin{equation} \label{CAslow}
\left\{
\begin{aligned}
&\epsilon\dfrac{dX}{d\tau} &=& -\beta T^*X +\chi_1 F_X^1\left(\mathbf{v},\mathbf{z}\right)
+\chi_2 F_X^2\left(\mathbf{v},\mathbf{z}\right)
+\chi_3 F_X^3\left(\mathbf{L}\right)
% \chi_i F_X^i\left(\mathbf{L},\mathbf{K},\mathbf{v},\mathbf{z}\right)
+ O(\epsilon)\\
&\epsilon\dfrac{dY}{d\tau} &=& \beta(S^* - T^* -k I^*)X - (m+\beta k T^*)Y 
+\chi_1 F_Y^1\left(\mathbf{v},\mathbf{z}\right)
+\chi_2 F_Y^2\left(\mathbf{v},\mathbf{z}\right)
+\chi_5 F_Y^5\left(\mathbf{v},\mathbf{z}\right)
%+ \sum_{i\in\{1,2,4,5\}} \chi_i F_Y^{i}\left(\mathbf{K},\mathbf{v},\mathbf{z}\right)
+ O(\epsilon)\\
&\epsilon\dfrac{dL_i}{d\tau} &=& -m L_i + \chi_3 F_{L_i} \left(\mathbf{v},\mathbf{z}\right)+ O\left(\epsilon\right)\\
&\epsilon\dfrac{dv_i}{d\tau} &=& -\xi v_i + O(\epsilon)\\
&\dfrac{dz_i}{d\tau} &=& F_{z_i}(X,Y,\mathbf{L},\mathbf{v},\mathbf{z})+O(\epsilon)
\end{aligned}
\right.
\end{equation}
Using the notation of the Theorem \ref{thm2}, we see that the fast variables is $y(\tau)=(X,Y,\mathbf{L},\mathbf{v})$ and the slow variable is $x(\tau)=\mathbf{z}(\tau)$.
The first step in applying the Tikhonov theorem is to take $\epsilon=0$ in \eqref{CAslow} and to show that the fast variable converge to an attractor $\mathbf{\phi}(\mathbf{z})$ which is parametrized by the slow variable.
\begin{lmm}\label{lemma_fast}
	Let $\epsilon=0$ in \eqref{CAslow}. Then there exist a function $\mathbf{\Phi}(\mathbf{z})=\left(X^*(\mathbf{z}),Y^*(\mathbf{z}),\chi_3\mathbf{L}^*(\mathbf{z}),0\right)$ 
	such that the solution $(X,Y,\mathbf{L},\mathbf{v},\mathbf{z})$ of $\eqref{CA}$ with any initial condition 
	$$(X,Y,\mathbf{L},\mathbf{v},\mathbf{z})(0)=(X_0,Y_0,\mathbf{L}_0,\mathbf{v}_0,\mathbf{z}_0)\in\R\times \R\times \R^n\times \R^n\times \R^n$$
	verifies $\mathbf{z}(t)=\mathbf{z}_0$ for all $t\geq 0$ and 
	$$\lim_{t\to+\infty}(X,Y,\mathbf{L},\mathbf{v})(t)=\mathbf{\Phi}(\mathbf{z}_0).$$
Moreover, $X^*$ and $Y^*$ are linear function of the $\chi_i$.
\end{lmm}
\begin{proof}
	Using the triangular structure of \eqref{CAslow} the idea is to compute the limits step by step of $\mathbf{v}$, $\mathbf{L}$, $X$ and $Y$ in this order. Here we make a quick formal computation by simply plugging the limits obtained at one step into the equation of the next step. It is easy to verified that this computation is justified and we omit it here for clarity.\\ 
	Since \eqref{CAslow} is equivalent to \eqref{CA} but in the slow motion, we take $\epsilon=0$ in \eqref{CA}.
	We have directly $\mathbf{z}(t)=\mathbf{z}_0$ for all $t\geq0$ and $v_i=e^{-\xi t} v_i(0) \to 0$ asymptotically as $t\to+\infty$.
	Remark that taking $v_i=0$ in the others equations leads to the simple change of variables : $I_i=I^* z_i$ and $J_i=T^*z_i$ that we can plug in \eqref{2.32}-\eqref{2.33}-\eqref{2.34} to simplify the explicit computations.\\
	Now we have the following asymptotic limits
	$$L_i(t)\to \chi_3 \frac{1}{m}F_{L_i}(0,\mathbf{z}_0)=\chi_3L_i^*(\mathbf{z}_0).$$
	Denoting $\mathbf{L}^*=\left(L_i^*\right)_i$ and plugging this into the equation of $X$ we have:
	$$X(t)\to -\frac{1}{\beta T^*} \left(
	\chi_1 F_X^1\left(0,\mathbf{z}_0\right)
	+\chi_2 F_X^2\left(0,\mathbf{z}_0\right)
	+\chi_3 F_X^3\left(\chi_3\mathbf{L}^*(\mathbf{z}_0)\right)
	\right)=X^*(\mathbf{z}_0).$$
	Remark that by linearity of the $F_X^i$ and the fact that $\chi_d^2=\chi_d$ for each $d$, we have the simpler formula
	\begin{equation}\label{Xstar}
	X^*(\mathbf{z}_0)=-\frac{1}{\beta T^*} \left(
	\chi_1 F_X^1\left(0,\mathbf{z}_0\right)
	+\chi_2 F_X^2\left(0,\mathbf{z}_0\right)
	+\chi_3 F_X^3\left(\mathbf{L}^*(\mathbf{z}_0)\right)
	\right).
	\end{equation}
	Finally, using the same arguments we get 
	$$Y(t)\to Y^*(\mathbf{z}_0)$$
	wherein we have note
	$$Y^*(\mathbf{z}_0)=\frac{1}{m+\beta k T^*}  \left(\beta(S^* - T^* -k I^*)X^*(\mathbf{z}_0)
	+\chi_1 F_Y^1\left(0,\mathbf{z}_0\right)
	+\chi_2 F_Y^2\left(0,\mathbf{z}_0\right)
	+\chi_5 F_Y^5\left(0,\mathbf{z}_0\right)
	\right).$$
\end{proof}
Now, we take $\epsilon=0$ in \eqref{CAslow} and we fixe
\begin{equation}\label{CAslow_fastpart}
(X,Y,\mathbf{L},\mathbf{v})(\tau)=\mathbf{\Phi}(\mathbf{z}(\tau)).
\end{equation} 
Then the $2+2N$ first equations  are satisfied and the $N$ last equations give the {\it slow system}
\begin{equation}\label{CAslow_slowpart}
\dfrac{dz_i}{d\tau} = F_{z_i}(X^*(\mathbf{z}),Y^*(\mathbf{z}),\mathbf{L}^*(\mathbf{z}),0,\mathbf{z}).
\end{equation}
It's important to note that, since $\mathbf{v} = 0$ \eqref{CAslow_slowpart} then \eqref{2.15} gives $\sum_{i = 1}^N z_i = 1$ by the formula $I_i = I^*z_i$. 
Hence $z_i$ reflects the frequency of strain $i$  for all $i$. Remark that we have also $J_i=T^*z_i$.\\
The Theorem \ref{thm2} imply that the solutions of \eqref{CAslow_slowpart} together with \eqref{CAslow_fastpart} gives a good approximation of the original system \eqref{CAslow} for a small enough but positive $\epsilon$.
Coming back to the original variables of the SIS system, we deduce the following result on error estimate, whose proof will be given in section \ref{sec3.5}.
\begin{lmm}\label{lmm_error}
Let $T>0$ be fixed. There exists $\epsilon_0>0$ and $C_T>0$ such that for any $\epsilon\in(0,\epsilon_0)$ we have
for any solution of $\left(S,(I_i)_i,(I_{ij})_{ij}\right)_{i,j}$ of \eqref{2.1}  and $(z_i)_i$ of \eqref{CAslow_slowpart}
\begin{equation} \label{error_lemma}
	\left|S\left(\dfrac{\tau}{\epsilon} \right) - S^*\right| 
	+\sum_{i=1}^{N}\left|I_i\left(\dfrac{\tau}{\epsilon} \right)-I^*z_i(\tau)\right| 
	+ \sum_{i,j = 1}^{N}\left|I_{ij}\left(\dfrac{\tau}{\epsilon}\right) - k\dfrac{I^*T^*}{S^*}z_i\left(\tau\right)z_j\left(\tau\right) \right| \leq \epsilon C_T ,
	\end{equation}
\end{lmm}
\begin{proof}
See section \ref{sec3.5}.
\end{proof}
It remains to compute explicitly the slow system \eqref{CAslow_slowpart}. The following lemma shows that it suffices to compute independently the system for each perturbation, that is $\mathcal{A}=\{d\}$ for $d\in\{1,2,3,4,5\}$. The case of a general $\mathcal{A}$ is simply  a sum over  simple cases thanks to the following result. 
\begin{lmm}\label{glue_lemma}
Let $\mathcal{A}\subset \{1,\cdots,5\}$. Recall that $\chi_d=1$ if  $d\in\mathcal{A}$ and $\chi_d=0$ if $d\notin\mathcal{A}$.  The functions $F_{z_i}$ for $i=1,\cdots,N$ 
in \eqref{CAslow_slowpart} read
$$F_{z_i}(X^*(\mathbf{z}),Y^*(\mathbf{z}),\mathbf{L}^*(\mathbf{z}),0,\mathbf{z})=\sum_{d=1}^5 \chi_d z_i f_{z_i}^d\left(\mathbf{z}\right),$$
where the functions $f_{z_i}^d$ do not depend on $\chi_d$.\\
In particular, if $\mathcal{A}=\{d\}$ for some $d\in\{1,2,3,4,5\}$, then 
$$F_{z_i}(X^*(\mathbf{z}),Y^*(\mathbf{z}),\mathbf{L}^*(\mathbf{z}),0,\mathbf{z})= z_i f_{z_i}^d\left(\mathbf{z}\right).$$
\end{lmm}
\begin{proof}
Taking $v_i=0$ in \eqref{eq:Fzi} we see that there is two constant $C_X$ and $C_Y$ such that 
\begin{equation*}
F_{z_i}(X^*(\mathbf{z}),Y^*(\mathbf{z}),\mathbf{L}^*(\mathbf{z}),0,\mathbf{z})= z_i \left(C_X X^*(\mathbf{z}),+C_Y Y^*(\mathbf{z})+\begin{pmatrix}
0 & 1
\end{pmatrix} P \mathcal{M}_{\mathcal{A}} P^{-1}\begin{pmatrix} 0\\1\end{pmatrix}\right) + \begin{pmatrix}
0 & 1
\end{pmatrix} \chi_3\gamma P^{-1}\begin{pmatrix}
0\\ L^*_i\left(\mathbf{z}\right)
\end{pmatrix} .
\end{equation*}
Firstly, as it is shown in the proof of the lemma \ref{lemma_fast}, the expression of $X^*$ and $Y^*$ are  both a linear combination of the $\chi_d$.\\ \\
Secondly, recalling that we have at this step $I_i=I^*z_i$, $J_i=T^* z_i$, $\mathbf{L} =\chi_3 \mathbf{L}^*$ and, in particular, $\chi_3^2=\chi_3$. Plugging this in \eqref{CA}, we see that the matrix $\mathcal{M}_{\mathcal{A}}$ is also a linear combination of the $\chi_d$ :
$$\mathcal{M}_{\mathcal{A}}=\sum_{d\in\mathcal{A}} \mathcal{M}_{\{d\}}=\sum_{d\in\{1,2,3,4,5\}} \chi_d\mathcal{M}_{d}.$$
denoting  $m_d \left(\mathbf{z}\right)=\begin{pmatrix}
0 & 1
\end{pmatrix} P^{-1} \mathcal{M}_{\{d\}} P\begin{pmatrix} 0\\1\end{pmatrix}$, this yields to:
\begin{equation}\label{Mchii} 
\begin{pmatrix}
0 & 1
\end{pmatrix} P^{-1} \mathcal{M}_{\mathcal{A}} P\begin{pmatrix} 0\\1\end{pmatrix}
=\sum_{d\in\{1,2,3,4,5\}} \chi_d m_d \left(\mathbf{z}\right).\end{equation}
Thirdly, plugging $I_i = I^*z_i$ and $J_i = T^*z_i$, for all $i$ in \eqref{2.32} we prove that
\begin{equation*}
L_i^*\left(\mathbf{z}\right) = \dfrac{1}{2m}\beta k I^*T^*z_i\sum_{j=1}^{N}\left( u_{ij} +  u_{ji}\right)z_j.
\end{equation*}
Actually, this value $L^*\left(\mathbf{z}\right)$ is exact as in \eqref{L*} computed in section \ref{sec3.3}.\\
The result follows directly from the three previous points.
\end{proof}
In the next section \ref{sec3}, these functions $f_{z_i}^d$ are explicitly computed for any $d$.
%%%%%%%%%%%%%%%%%%%%%%%%%%%%%%%%%%%%%%%%%%%%%%%%%%%%%%%%%%%%%%%
\subsection{Main Results}\label{sec2.5}
%%%%%%%%%%%%%%%%%%%%%%%%%%%%%%%%%%%%%%%%%%%%%%%%%%%%%%%%%%%%%%%
In the earlier study \cite{Madec2}  we computed the slow dynamics for $\mathcal{A}=\{5\}$, that is for perturbation in $k_{ij}=k+\epsilon \alpha_{ij}$ only, i.e. for strains varying only in their co-colonization susceptibility interactions. We found that the slow system obeys a {\em replicator equation} which has the from
\begin{equation}\label{replicator0}
	\begin{cases}
	\dot{z_i}=\Theta z_i\left( \left(\Lambda \mathbf{z}\right)_i-\mathbf{z}^T \Lambda \mathbf{z}\right),\; i=1,\cdots,N,\\
	{\displaystyle \sum_{i=1}^N z_i =1}
	\end{cases}
\end{equation}
where $\Theta$ is a positive constant depending on the parameters of the \textit{neutral system} and $\Lambda=\left(\lambda_i^j\right)_{i,j}$ is the $N \times N$ matrix of pairwise invasion fitness among strains where the term of line $i$ and column $j$ was
\begin{equation*}
\lambda_{i}^j= \dfrac{I^*}{D^*}\left(\alpha_{ji}-\alpha_{ij}\right)+\left(\alpha_{ji}-\alpha_{jj}\right).
\end{equation*}
In this present article,  we show that the system \eqref{replicator0} is true for any type of perturbation. 
The change is that the constant $\Theta$ and the pairwise fitness $\lambda_i^j$ depend on the multiple trait variations  which occur in the system.
From the Lemma \ref{glue_lemma}, we infer in particular that the $\lambda_i^j$ are just a linear combination of the different perturbations. This implies that the pairwise invasion fitness between any two strains is an explicit weighted sum over all fitness dimensions where the two strains vary. 
More precisely, the main result of this article is as follows.\\
Let $\mathcal{A}\subset\{1,2,3,4,5\}$. Using the notations in the previous section, we prove in the \ref{sec3} that \eqref{CAslow_slowpart} reads.
\begin{equation} \label{2.44}
\boxed{
	\begin{aligned}
	\dfrac{dz_i}{d\tau} = 
	\Theta_1z_i\left(b_i - \sum_{j=1}^{N}b_j z_j\right) 
	&+ \Theta_2z_i\left(-\nu_i + \sum_{j=1}^{N}\nu_jz_j\right) 
	+\Theta_3z_i\left[-\sum_{j=1}^{N}( u_{ij}+  u_{ji})z_j + \sum_{j,l=1}^{N}( u_{jl} +  u_{lj})z_l z_j \right]\\
	& 
	+ \Theta_4 z_i \left[\sum\limits_{j = 1}^N (\omega_{ij}^i-\omega_{ji}^j)z_j\right]
	+\Theta_5 z_i\left[\sum\limits_{j = 1}^N\left(\dfrac{T^*}{D^*}\alpha_{ji} - \dfrac{I^*}{D^*}\alpha_{ij}\right)z_j - \sum\limits_{j,l = 1}^N \alpha_{jl}z_j z_l\right]
	\end{aligned}}
\end{equation}
where  
\begin{equation}\label{Theta_i}
\Theta_1 = \chi_1\dfrac{2\beta S^*{T^*}^2}{|P|}, \quad
\Theta_2 = \chi_2\dfrac{\gamma I^*\left(I^* + T^*\right)}{|P|},\quad
\Theta_3 = \chi_3\dfrac{\gamma T^*D^*}{|P|},\quad
\Theta_4 = \chi_4\dfrac{2mT^*D^*}{|P|},
\quad
\Theta_5 = \chi_5\dfrac{\beta T^*I^*{D^*}}{|P|}.
\end{equation}
Naturally, if $\mathcal{A}=\emptyset$, \eqref{2.44} becomes simply $\dfrac{dz_i}{d\tau} = 0$. 
 Otherwise, if $\mathcal{A}\neq\emptyset$, it is useful to rewrite  \eqref{2.44} using the pairwise invasion fitness between strains in \eqref{replicator0}.
 Define 
\begin{equation}\label{theta_i}
\Theta = \Theta_1 + \Theta_2 + \Theta_3 + \Theta_4 +\Theta_5\quad \text{ and }\theta_i = \dfrac{\Theta_i}{\Theta}.
\end{equation} we see that  $\theta_i\geq 0$ for each $i=1,2,3,4,5$ and  $\theta_1 +\theta_2 +\theta_3 +\theta_4 +\theta_5= 1$. For completeness, if $\mathcal{A}=\emptyset$  then we set $\Theta=1$.  Using these notations, we obtain our main result.
\begin{thm}
Consider the system of equations
	\begin{equation} \label{2.47}
	\boxed{
		\left\{
		\begin{aligned}
		&\dot{z_i}=\Theta z_i\left( \left(\Lambda \mathbf{z}\right)_i-\mathbf{z}^T \Lambda \mathbf{z}\right),\; i=1,\cdots,N,\\
		&z_1 + z_2 + \dots + z_N = 1.
		\end{aligned} 
		\right.}
	\end{equation}
where $\Lambda$ is the square matrix of size $N \times N$ whose coefficients $(i;j)$ are the pairwise invasion fitnesses $\lambda^j_i$ which satisfy
\begin{equation} \label{2.48}
\boxed{\begin{aligned}
\lambda^j_i = \theta_1\left(b_i -b_j \right) &+ \theta_2\left(-\nu_i + \nu_j\right) + \theta_3\left(- u_{ij} -  u_{ji} + 2 u_{jj}\right) \\
&+\theta_4\left(\omega_{ij}^i-\omega_{ji}^j\right)
+ \theta_5 \left(\mu\left(\alpha_{ji} - \alpha_{ij}\right) + \alpha_{ji} - \alpha_{jj}\right)
.
\end{aligned}
}
\end{equation}
with $\mu = \dfrac{I^*}{D^*}$. \\
Then, for any initial values of \eqref{2.1}, for each $\tau_0 > 0$, $T > \tau_0$ arbitrarily and independent on $\epsilon$, there is  $\epsilon_0 > 0$, $C > 0$  and a  vector of positive coefficients $\mathbf{z}_0\in\mathbb{R}^N$ verifying  $\sum_{i=1}^N \mathbf{z}_{0,i} =1$, such that $\forall \epsilon < \epsilon_0$
\begin{equation}
\left|S\left(\dfrac{\tau}{\epsilon} \right) - S^*\right| 
+ \sum_{i=1}^{N}\left|I_i\left(\dfrac{\tau}{\epsilon} \right) - I^*z_i(\tau)  \right| 
+ \sum_{i,j=1}^{N}\left| I_{ij}\left(\dfrac{\tau}{\epsilon} \right) - k\dfrac{I^*T^*}{S^*}z_i(\tau)z_j(\tau) \right| \leq \epsilon C , \quad \forall \tau \in \left(\tau_0,T\right).
\end{equation}
where $S$, $(I_1, I_2, \dots, I_N)$, $\left(I_{ij}\right)_{i,j\in\{1,\dots,N\}}$ is the solution of \eqref{2.1} and $(z_1, z_2, \dots, z_N)$ is the solution of reduced system \eqref{2.47} together with $\mathbf{z}(0)=\mathbf{z}_0$.
\end{thm}
This system \eqref{2.47} is a  general replicator system, which is studied in \cite{Hofbauer}.\\
We have two remarks on $\lambda_i^j$ in \eqref{2.48}. The first is that, each coefficient $\theta_i$, $i\in\{1,2,3,4,5\}$ measures the  weight of each trait perturbation on pairwise invasion fitness. Thus, each $\lambda_{i}^j$ is a weighted average of the perturbations. Secondly, the pairwise invasion fitnesses play an important role in predicting collective dynamics, since $\lambda_{i}^j$ is the pairwise invasion fitness between strains $i$ and $j$, describing the quantitative initial growth rate of $i$ invading an equilibrium set by $j$ alone. In a 2-strain system, recall the final outcome results depend on the signs of the these mutual coefficients between the strains (Table \ref{tab:table2}), mentioned and used in \cite{Gjini,Gjini1,Madec2}.
\begin{table}[H]
	\caption{\textbf{From 2-strain invasion dynamics to collective multi-strain dynamics.} Each pair of strains in the system falls in one of 4 classes, according to $\lambda^2_1$ and $\lambda^1_2$ in \eqref{2.47}: either competitive exclusion of $1$, competitive exclusion of $2$, coexistence, or bistability. The $N$-strain mutual invasion network drives competitive dynamics over long time.}\label{tab:table2}
	\begin{center}
		%\scalebox{1}{
			\begin{tabular}{p{4cm}  p{4cm} p{3cm} p{5cm}} 
				\hline
				\\
				Mutual invasion $\left(\lambda^2_1, \lambda^1_2\right)$ & Pairwise Outcome  & $N$-strain network & Strain freq.
				\\
				\hline
				\textcolor{red}{$(+,+)$} & Stable coexistence   & \multirow{ 4}{*}{\centering \includegraphics[width=0.09 \textwidth]{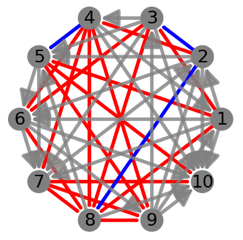}} & \multirow{4}{*}{$\dot{z_i}=\Theta z_i\left( \left(\Lambda \mathbf{z}\right)_i-\mathbf{z}^T \Lambda \mathbf{z}\right)$}\\
				\textcolor{gray}{$(+,-)$} & Exclusion of type 1 & & \multirow{4}{*}{$i=1...N$}\\
				\textcolor{gray}{$(-,+)$} & Exclusion of type 2& &\\
				\textcolor{blue}{$(-,-)$} & Bistability& &\\
				\hline
				\multicolumn{4}{c}{$\lambda^j_i = \theta_1\left(b_i -b_j \right) + \theta_2\left(-\nu_i + \nu_j\right) + \theta_3\left(- u_{ij} -  u_{ji} + 2 u_{jj}\right)
+\theta_4\left(\omega_{ij}^i-\omega_{ji}^j\right)+ \theta_5 \left(\mu\left(\alpha_{ji} - \alpha_{ij}\right) + \alpha_{ji} - \alpha_{jj}\right)$}
\\
\hline
		\end{tabular}
		% }
	\end{center}
\end{table}
In the next section, we present explicitly all the necessary computations and we also prove the lemma for the error estimate \ref{lmm_error}.
%%%%%%%%%%%%%%%%%%%%%%%%%%%%%%%%%%%%%%%%%%%%%%%%%%%%%%%%%%%
\section{Proofs and explicit computations}\label{sec3}
Initially, let us  recall the following definitions.
\begin{itemize}
	\item $S$: total proportion of susceptible hosts
	\item $T$: the total proportion of infected hosts (prevalence of colonization)
	\item $I_i$: the proportion of hosts singly-colonized by strain-$i$ 
	\item $I_{ij}$: the proportion of hosts co-colonized by strain-$i$ then strain-$j$ (Including $I_{ii}$).
\end{itemize} 
%%%%%%%%%%%%%%%%%%%%%%%%%%%%%%%%%%%%%%%%%%%%%%%%%%%%%%%%%%%%%%%%%ù
%%%%%%%%%%%%%%%%%%%%%%%%%%%%%%%%%%%%%%%%%%%%%%%%%%%%%%%%%%%%%%%%%ù
%%%%%%%%%%%%%%%%%%%%%%%%%%%%%%%%%%%%%%%%%%%%%%%%%%%%%%%%%%%%%%%%%ù
\subsection{$\mathcal{A}=\{1\}$. Perturbations only in transmission rates $\beta_i$} \label{sec3.1}
%%%%%%%%%%%%%%%%%%%%%%%%%%%%%%%%%%%%%%%%%%%%%%%%%%%%%%%%%%%%%%%%%ù
%%%%%%%%%%%%%%%%%%%%%%%%%%%%%%%%%%%%%%%%%%%%%%%%%%%%%%%%%%%%%%%%%ù
%%%%%%%%%%%%%%%%%%%%%%%%%%%%%%%%%%%%%%%%%%%%%%%%%%%%%%%%%%%%%%%%%ù
Here we compute the functions $f_{z_i}^1$.
In \eqref{CA}, take $\epsilon=0$,  $\chi_1=1$ and $\chi_d=0$ for $d>1$. It comes
\begin{equation} \label{CA1}
\left\{
\begin{aligned}
&\dfrac{dX}{dt} &=& -\beta T^*X + F_X^1\left(\mathbf{v},\mathbf{z}\right)
% \chi_i F_X^i\left(\mathbf{L},\mathbf{K},\mathbf{v},\mathbf{z}\right)
\\
&\dfrac{dY}{dt} &=& \beta(S^* - T^* -k I^*)X - (m+\beta k T^*)Y 
+ F_Y^1\left(\mathbf{v},\mathbf{z}\right)
%+ \sum_{i\in\{1,2,4,5\}} \chi_i F_Y^{i}\left(\mathbf{K},\mathbf{v},\mathbf{z}\right)
\\
&\dfrac{dL_i}{dt} &=& -m L_i \\
&\dfrac{dv_i}{dt} &=& -\xi v_i \\
&\dfrac{dz_i}{dt} &=& 0
\end{aligned}
\right.
\end{equation}
Following the notation of the lemma \ref{lemma_fast}, we obtain
 that the solution $(X,Y,\mathbf{L},\mathbf{v},\mathbf{z})$ of $\eqref{CA1}$ with the initial condition 
$(X,Y,\mathbf{L},\mathbf{v},\mathbf{z})(0)=(X_0,Y_0,\mathbf{L}_0,\mathbf{v}_0,\mathbf{z}_0)\in\R\times \R\times \left(\R^n\right)^3$ verifies

$$\lim_{t\to+\infty}(X,Y,\mathbf{L},\mathbf{v})(t)=\left(X^*(\mathbf{z}_0),Y^*(\mathbf{z}_0),0,0\right).$$
for some functions $X^*(\mathbf{z})$ and $Y^* (\mathbf{z})$ which remains to be compute.\\
Replacing $\mathbf{L}$ and $\mathbf{v}$ by $0$ in the two first equation of \eqref{CA1} yields 
\begin{equation} \label{CA1_details1}
\left\{
\begin{aligned}
&\dfrac{dX}{dt} &=& -\beta T^*X + F_X^1\left(0,\mathbf{z}\right)
% \chi_i F_X^i\left(\mathbf{L},\mathbf{K},\mathbf{v},\mathbf{z}\right)
\\
&\dfrac{dY}{dt} &=& \beta(S^* - T^* -k I^*)X - (m+\beta k T^*)Y 
+ F_Y^1\left(0,\mathbf{z}\right)
\end{aligned}
\right.
\end{equation}
Note  that $\mathbf{v}=0$ implies that the change of variables \eqref{2.14} reads simply 
$$I_i = I^*z_i,\quad J_i =T^* z_i.$$
The quantities  $F_X^1\left(0,\mathbf{z}\right)$ and $F_Y^1\left(0,\mathbf{z}\right)$ are then easily deducting from  \eqref{2.32}
\begin{equation}
F_X^1\left(0,\mathbf{z}\right)=\beta S^*T^*\sum_{i=1}^N b_j z_j, \qquad F_Y^1\left(0,\mathbf{z}\right)=\beta\left(S^*-k I^*\right)T^*\sum_{i=1}^N b_j z_j. 
\end{equation}
Plugging this in \eqref{CA1_details1}, we obtain
$$X^*(\mathbf{z})= S^*  \sum_{i=1}^N b_j z_j$$
and then
\begin{equation*}
%Y^*(\mathbf{z}) = \dfrac{1}{m + \beta k T^*} \left(\beta\left(S^*-T^*-kI^*\right)X^*(\mathbf{z})+(S^*-kI^*)D^*\sum_{i=1}^N b_j z_j\right)
Y^*(\mathbf{z}) = \beta\cdot\dfrac{{S^*}^2 - k I^* S^* - kI^* T^*}{m+\beta k T^*}\sum_{i=1}^N b_j z_j = \beta\cdot\dfrac{{S^*}^2 - k I^* }{m+\beta k T^*}\sum_{i=1}^N b_j z_j.
\end{equation*}
%which gives
%\begin{equation*}
%Y^*(\mathbf{z})=\frac{1}{m+\beta k T^*} \left(\beta\left(S^*-T^*-kI^*\right) X^*(\mathbf{z}) + (S^*-kI^*) D^*\sum_{i=1}^N b_j z_j\right).
%\end{equation*}
%Using that $I^*=\frac{m T^*}{m+\beta k T^*} $ we find 
%\begin{equation*}
%Y^*(\mathbf{z})=\left((S^*-kI^*)\left(\frac{S^*}{T^*}+1\right)-S^*\right) \dfrac{I^*D^*}{m T^*} \sum_{i=1}^N b_j z_j.
%\end{equation*}
Now, \eqref{eq:Fzi} with the notation of the lemma \ref{glue_lemma}   gives 
\begin{equation}
\begin{aligned}
f_{z_i}^1(\mathbf{z})=
&-\dfrac{\beta}{|P|} \left(k{I^*}^2 - T^*D^* +2{T^*}^2\right)S^* \sum_{i=1}^N b_j z_j 
+ \dfrac{\beta}{|P|} \dfrac{\beta k {T^*}^2}{m+\beta k T^*} \left({S^*}^2 - k I^* \right) \sum_{i=1}^N b_j z_j \\
&- \dfrac{\beta}{|P|}\left(k{I^*}^2T^*\sum_{i=1}^N b_j z_j + b_i D^*T^*S^* - 2b_i {T^*}^2 S^* -kb_i{T^*}^2I^*\right).
\end{aligned}
\end{equation}
Denote
\begin{equation}
\begin{aligned}
&G &=&  -D^*T^*S^* + 2 {T^*}^2 S^* + k{T^*}^2I^* \\
&H &=&  -\left(k{I^*}^2 - T^*D^* +2{T^*}^2\right)S^*+ D^*\left({S^*}^2 - k I^* \right) - k{I^*}^2T^*
\end{aligned}
\end{equation}
then $G = -H = 2{T^*}^2S^* > 0$, by straightforward computations. Setting $\Theta_1 = \dfrac{2\beta {T^*}^2S^*}{|P|}>0$, we have 
\begin{equation} \label{fzi1}
\begin{aligned}
f_{z_i}^1(\mathbf{z})=\Theta_1 \left(b_i-\sum_{j=1}^{N}b_j z_j \right).
\end{aligned}
\end{equation}
 It follows that the slow system \eqref{CAslow_slowpart} reads 
\begin{equation} \label{3.7}
\boxed{\dfrac{dz_i}{d\tau} = \Theta_1 z_i \left(b_i-\sum_{j=1}^{N}b_j z_j \right), \qquad 1 \leq i \leq N.}
\end{equation}
Now we will show the simple computations showing that this system is exactly on the form of the replicator equation \eqref{2.47}. It is clear that the set    $\left\{\mathbf{z}\in[0,1]^N,\;\sum_{i=1}^{N}z_i = 1\right\}$, is conserved for \eqref{3.7}. Hence, \eqref{fzi1} may be rewrite as 
\begin{equation} \label{3.8}
f_{z_i}^1(\mathbf{z})=\Theta_1 \left(\sum_{j=1}^{N}(b_i-b_j) z_j \right). 
\end{equation}
Denoting pairwise invasion fitness between strains $i$ and $j$, $i$ invading in an equilibrium set by $j$, $\lambda_i^j = (b_i-b_j) $ and $\Lambda=(\lambda_i^j)$, we have 
\begin{equation}
f_{z_i}^1(\mathbf{z})=\Theta_1 \left(\Lambda \mathbf{z}\right)_i. 
\end{equation}
Finally, from $\Lambda^T=-\Lambda$ we see that $\mathbf{z}^T \Lambda \mathbf{z}=0$ which leads to the (artificial) representation of \eqref{3.7} : 
\begin{equation} 
\boxed{\dfrac{dz_i}{d\tau} = \Theta_1 z_i \left(\left(\Lambda \mathbf{z}\right)_i- \mathbf{z}^T \Lambda \mathbf{z}\right), \qquad 1 \leq i \leq N.}
\end{equation}
which is nothing but the slow system \eqref{2.47} with $\lambda_i^j=b_i-b_j$.
%%%%%%%%%%%%%%%%%%%%%%%%%%%%%%%%%%%%%%%%%%%%%%%%%%%%%%%%%%%%%%%%%ù
%%%%%%%%%%%%%%%%%%%%%%%%%%%%%%%%%%%%%%%%%%%%%%%%%%%%%%%%%%%%%%%%%ù
%%%%%%%%%%%%%%%%%%%%%%%%%%%%%%%%%%%%%%%%%%%%%%%%%%%%%%%%%%%%%%%%%ù
\subsection{$\mathcal{A}=\{2\}$. Perturbations only in clearance rates of single colonization $\gamma_i$} \label{sec3.2}
%%%%%%%%%%%%%%%%%%%%%%%%%%%%%%%%%%%%%%%%%%%%%%%%%%%%%%%%%%%%%%%%%ù
%%%%%%%%%%%%%%%%%%%%%%%%%%%%%%%%%%%%%%%%%%%%%%%%%%%%%%%%%%%%%%%%%ù
%%%%%%%%%%%%%%%%%%%%%%%%%%%%%%%%%%%%%%%%%%%%%%%%%%%%%%%%%%%%%%%%%ù
Similarly to the case $\mathcal{A} = \left\{1\right\}$, we compute the functions $f_{z_i}^2$.
In \eqref{CA}, take $\epsilon=0$,  $\chi_2=1$ and $\chi_d = 0$ for $d \neq 2$. It comes
\begin{equation} \label{CA2}
\left\{
\begin{aligned}
&\dfrac{dX}{dt} &=& -\beta T^*X + F_X^2\left(\mathbf{v},\mathbf{z}\right)
% \chi_i F_X^i\left(\mathbf{L},\mathbf{K},\mathbf{v},\mathbf{z}\right)
\\
&\dfrac{dY}{dt} &=& \beta(S^* - T^* -k I^*)X - (m+\beta k T^*)Y 
+ F_Y^2\left(\mathbf{v},\mathbf{z}\right)
%+ \sum_{i\in\{1,2,4,5\}} \chi_i F_Y^{i}\left(\mathbf{K},\mathbf{v},\mathbf{z}\right)
\\
&\dfrac{dL_i}{dt} &=& -m L_i \\
&\dfrac{dv_i}{dt} &=& -\xi v_i \\
&\dfrac{dz_i}{dt} &=& 0
\end{aligned}
\right.
\end{equation}
Following the notation of the lemma \ref{glue_lemma}, we obtain that the solution $(X,Y,\mathbf{L},\mathbf{v},\mathbf{z})$ of $\eqref{CA2}$ with the initial condition $(X,Y,\mathbf{L},\mathbf{v},\mathbf{z})(0)=(X_0,Y_0,\mathbf{L}_0,\mathbf{v}_0,\mathbf{z}_0)\in\R\times \R\times \left(\R^n\right)^3$ verifies

$$\lim_{t\to+\infty}(X,Y,\mathbf{L},\mathbf{v})(t)=\left(X^*(\mathbf{z}_0),Y^*(\mathbf{z}_0),0,0\right).$$
for some functions $X^*(\mathbf{z})$ and $Y^* (\mathbf{z})$ which remains to be compute.\\
Replacing $\mathbf{L}$, $\mathbf{K}$ and $\mathbf{v}$ by $0$ in the two first equation of \eqref{CA2} yields 
\begin{equation} \label{CA2_details1}
\left\{
\begin{aligned}
&\dfrac{dX}{dt} &=& -\beta T^*X + F_X^2\left(0,\mathbf{z}\right)
% \chi_i F_X^i\left(\mathbf{L},\mathbf{K},\mathbf{v},\mathbf{z}\right)
\\
&\dfrac{dY}{dt} &=& \beta(S^* - T^* -k I^*)X - (m+\beta k T^*)Y 
+ F_Y^2\left(0,\mathbf{z}\right)
\end{aligned}
\right.
\end{equation}
%Note  that $\mathbf{v}=0$ implies that the change of variables \eqref{2.14} reads simply 
%$$I_i = I^*z_i,\quad J_i =D^* z_i.$$
The quantities  $F_X^2\left(0,\mathbf{z}\right)$ and $F_Y^2\left(0,\mathbf{z}\right)$ are then easily deducting from  \eqref{2.32} 
\begin{equation}
F_X^2\left(0,\mathbf{z}\right)=-  \gamma I^*\sum_{i=1}^{N}\nu_i z_i, \qquad F_Y^2\left(0,\mathbf{z}\right)=- \gamma I^*\sum_{i=1}^{N}\nu_iz_i.
\end{equation}
Plugging this in \eqref{CA2_details1}, we obtain
$$X^*(\mathbf{z})= - \dfrac{\gamma I^*}{\beta^2 T^*}\sum_{i=1}^N \nu_i z_i $$
and then
\begin{equation*}
Y^*(\mathbf{z}) = \dfrac{ \gamma I^*(kI^* - S^*)}{T^*(m + \beta k T^*)} \sum_{i=1}^N\nu_i z_i.
\end{equation*}
Now, \eqref{eq:Fzi} with the notation of the lemma \ref{glue_lemma}   gives 
\begin{equation} \label{3.14}
\begin{aligned}
f_{z_i}^2(\mathbf{z})=
\dfrac{\gamma}{|P|}\left[-\nu_i I^*(I^* + T^*) +\dfrac{ I^*\left(kI^* T^* - D^* +2T^* \right)}{T^*}\sum_{j=1}^N\nu_j z_j\right].
\end{aligned}
\end{equation}
By straightforward computations we can verify that
\begin{equation}
\begin{aligned}
\left(-kI^* T^* + D^* -2T^* \right) + T^*(I^* + T^*) = 0.
\end{aligned}
\end{equation} 
Setting $\Theta_2 = \dfrac{\gamma I^*\left(I^* + T^*\right)}{|P|}>0$, we have 
\begin{equation} \label{fzi2}
\begin{aligned}
f_{z_i}^2(\mathbf{z})=\Theta_2 \left(-\nu_i+\sum_{j=1}^{N}\nu_j z_j \right).
\end{aligned}
\end{equation}
It follows that the slow system \eqref{CAslow_slowpart} reads 
\begin{equation} \label{3.17}
\boxed{\dfrac{dz_i}{d\tau} = \Theta_2 z_i \left(-\nu_i+\sum_{j=1}^{N}\nu_j z_j \right), \qquad 1 \leq i \leq N.}
\end{equation}
By the same arguments in section \ref{sec3.1}, we can show the simple computations showing that this system is exactly on the form of the replicator equation \eqref{2.47}. Denoting the pairwise invasion fitness $$\lambda_i^j = (-\nu_i+\nu_j) $$ and $\Lambda=(\lambda_i^j)$, we have 
\begin{equation} \label{3.18}
f_{z_i}^2(\mathbf{z})=\Theta_2 \left(\Lambda \mathbf{z}\right)_i. 
\end{equation}
Finally, from $\Lambda^T=-\Lambda$ we see that $\mathbf{z}^T \Lambda \mathbf{z}=0$ which leads to the (artificial) representation of \eqref{3.17} : 
\begin{equation} 
\boxed{\dfrac{dz_i}{d\tau} = \Theta_2 z_i \left(\left(\Lambda \mathbf{z}\right)_i- \mathbf{z}^T \Lambda \mathbf{z}\right), \qquad 1 \leq i \leq N.}
\end{equation}
which is nothing but slow system \eqref{2.47} with $\lambda_i^j= - \nu_i + \nu_j$.
\subsection{$\mathcal{A}=\{3\}$. Perturbations only in clearance rates of co-colonization $\gamma_{ij}$} \label{sec3.3}
Similarly to the case $\mathcal{A} = \left\{1\right\}$, we compute the functions $f_{z_i}^3$.
In \eqref{CA}, take $\epsilon=0$,  $\chi_3=1$ and $\chi_d=0$ for $d \neq 3$. It comes
\begin{equation} \label{CA3}
\left\{
\begin{aligned}
&\dfrac{dX}{dt} &=& -\beta T^*X + F_X^3\left(\mathbf{L}\right)
% \chi_i F_X^i\left(\mathbf{L},\mathbf{K},\mathbf{v},\mathbf{z}\right)
\\
&\dfrac{dY}{dt} &=& \beta(S^* - T^* -k I^*)X - (m+\beta k T^*)Y 
%+ \sum_{i\in\{1,2,4,5\}} \chi_i F_Y^{i}\left(\mathbf{K},\mathbf{v},\mathbf{z}\right)
\\
&\dfrac{dL_i}{dt} &=& -m L_i + F_{L_i}\left(\mathbf{v},\mathbf{z}\right)\\
&\dfrac{dv_i}{dt} &=& -\xi v_i \\
&\dfrac{dz_i}{dt} &=& 0
\end{aligned}
\right.
\end{equation}
Following the notation of the lemma \ref{glue_lemma}, we obtain that the solution $(X,Y,\mathbf{L},\mathbf{v},\mathbf{z})$ of $\eqref{CA3}$ with the initial condition $(X,Y,\mathbf{L},\mathbf{v},\mathbf{z})(0)=(X_0,Y_0,\mathbf{L}_0,\mathbf{v}_0,\mathbf{z}_0)\in\R\times \R\times \left(\R^n\right)^3$ verifies
$$\lim_{t\to+\infty}(X,Y,\mathbf{L},\mathbf{v})(t)=\left(X^*(\mathbf{z}_0),Y^*(\mathbf{z}_0),{\mathbf{L}}^*(\mathbf{z}_0),0,0\right).$$
for some functions $X^*(\mathbf{z})$, $Y^* (\mathbf{z})$ and ${\mathbf{L}}^*(\mathbf{z}_0)$ which remains to be compute.\\
Replacing  $\mathbf{v}$ by $0$ in the two first equation of \eqref{CA3} yields 
\begin{equation} \label{CA3_details1}
\left\{
\begin{aligned}
&\dfrac{dX}{dt} &=& -\beta T^*X + F_X^3\left(\mathbf{L}\right)
% \chi_i F_X^i\left(\mathbf{L},\mathbf{K},\mathbf{v},\mathbf{z}\right)
\\
&\dfrac{dY}{dt} &=& \beta(S^* - T^* -k I^*)X - (m+\beta k T^*)Y \\
&\dfrac{dL_i}{dt} &=& -mL_i + F_{L_i}\left(0,\mathbf{z}\right)
\end{aligned}
\right.
\end{equation}
%Note  that $\mathbf{v}=0$ implies that the change of variables \eqref{2.14} reads simply 
%$$I_i = I^*z_i,\quad J_i =D^* z_i.$$
The quantities  $F_{L_i}\left(0,\mathbf{z}\right)$ and $F_X^3\left(\mathbf{L}\right)$ are then easily deducting from  \eqref{2.32}.
\begin{equation}
F_X^3\left(\mathbf{L}\right) = -\gamma\sum_{i=1}^{N}L_i, \qquad F_{L_i}\left(0,\mathbf{z}\right)= \dfrac{1}{2}\beta k I^*T^*z_i\sum_{j=1}^{N}\left( u_{ij} +  u_{ji}\right)z_j.
\end{equation}
Plugging this in \eqref{CA3_details1}, we obtain
\begin{equation} \label{L*}
{L_i}^*(\mathbf{z})= \dfrac{1}{2m}\beta k I^*T^*z_i\sum_{j=1}^{N}\left( u_{ij} +  u_{ji}\right)z_j,
\end{equation}
then we deduce that
\begin{equation*}
X^*(\mathbf{z}) = - \dfrac{\gamma k I^*}{2m}\sum_{i,j = 1}^N( u_{ij} +  u_{ji})z_i z_j
\end{equation*}
and
\begin{equation*}
Y^*(\mathbf{z}) = - \dfrac{\beta\gamma k I^*(S^* - T^* - kI^*)}{2m(m+ \beta k T^*)}\sum_{i,j = 1}^N( u_{ij}+ u_{ji})z_i z_j .
\end{equation*}
Now, \eqref{eq:Fzi} with the notation of the lemma \ref{glue_lemma}   gives 
\begin{equation} \label{3.23}
\begin{aligned}
f_{z_i}^3(\mathbf{z})=
\dfrac{\gamma}{|P|} \left[\dfrac{\beta k I^* {T^*}^2}{m}\sum_{j,l=1}^{N}( u_{jl} -  u_{lj})z_l z_j
+\dfrac{\beta  k I^* {T^*}^2}{m}\sum_{j=1}^{N}( u_{ij} +  u_{ji})z_j\right].
\end{aligned}
\end{equation}
It's trivial to see that $\dfrac{\beta  k I^* {T^*}^2}{m} = T^*D^*$. Setting $\Theta_3 = \dfrac{\gamma T^*D^*}{|P|}>0$, we have 
\begin{equation} \label{fzi3}
\begin{aligned}
f_{z_i}^3(\mathbf{z})=\Theta_3\left[-\sum_{j=1}^{N}( u_{ij}+  u_{ji})z_j + \sum_{j,l=1}^{N}( u_{jl} +  u_{lj})z_l z_j \right].
\end{aligned}
\end{equation}
It follows that the slow system \eqref{CAslow_slowpart} reads 
\begin{equation} \label{3.25}
\boxed{\dfrac{dz_i}{d\tau} = \Theta_3z_i\left[-\sum_{j=1}^{N}( u_{ij}+  u_{ji})z_j + \sum_{j,l=1}^{N}( u_{jl} +  u_{lj})z_l z_j \right], \qquad 1 \leq i \leq N.}
\end{equation}
By the same arguments in section \ref{sec3.1}, we can show the simple computations showing that this system is exactly on the form of the replicator equation \eqref{2.47}. Denoting pairwise invasion fitness $$\lambda_i^j = - u_{ij} -  u_{ji} + 2 u_{jj} $$ and $\Lambda=(\lambda_i^j)$, we have 
\begin{equation} \label{3.26}
f_{z_i}^3(\mathbf{z})=\Theta_3 \left(\left(\Lambda \mathbf{z}\right)_i - \mathbf{z}^T\Lambda\mathbf{z}\right). 
\end{equation}
Finally, we see the (artificial) representation of \eqref{3.25} : 
\begin{equation} 
\boxed{\dfrac{dz_i}{d\tau} = \Theta_3 z_i \left(\left(\Lambda \mathbf{z}\right)_i- \mathbf{z}^T \Lambda \mathbf{z}\right), \qquad 1 \leq i \leq N.}
\end{equation}
which is nothing but replicator system \eqref{2.47} with $\lambda_i^j= - u_{ij} -  u_{ji} + 2 u_{jj}$.
%%%%%%%%%%%%%%%%%%%%%%%%%%%%%%%%%%%%%%%%%%%%%%%%%%%%%%%%%%%%%%%%%%%%%%%%%%%%%%%%%%%%%%%%%%%%%%%%
%%%%%%%%%%%%%%%%%%%%%%%%%%%%%%%%%%%%%%%%%%%%%%%%%%%%%%%%%%%%%%%%%%%%%%%%%%%%%%%%%%%%%%%%%%%%%%%%
\subsection{$\mathcal{A}=\{4\}$. 
Perturbations only in transmission coefficients from mixed co-colonization  $p^i_{ij}$} \label{sec3.4}
%%%%%%%%%%%%%%%%%%%%%%%%%%%%%%%%%%%%%%%%%%%%%%%%%%%%%%%%%%%%%%%%%%%%%%%%%%%%%%%%%%%%%%%%%%%%%%%%
%%%%%%%%%%%%%%%%%%%%%%%%%%%%%%%%%%%%%%%%%%%%%%%%%%%%%%%%%%%%%%%%%%%%%%%%%%%%%%%%%%%%%%%%%%%%%%%%
Similarly to the case $\mathcal{A} = \left\{1\right\}$, we compute the functions $f_{z_i}^4$.
In \eqref{CA}, take $\epsilon=0$,  $\chi_4=1$ and $\chi_d=0$ for $d \neq 4$. It comes
\begin{equation} \label{CA5}
\left\{
\begin{aligned}
&\dfrac{dX}{dt} &=& -\beta T^*X 
% \chi_i F_X^i\left(\mathbf{L},\mathbf{K},\mathbf{v},\mathbf{z}\right)
\\
&\dfrac{dY}{dt} &=& \beta(S^* - T^* -k I^*)X - (m+\beta k T^*)Y 
%+ \sum_{i\in\{1,2,4,5\}} \chi_i F_Y^{i}\left(\mathbf{K},\mathbf{v},\mathbf{z}\right)
\\
&\dfrac{dL_i}{dt} &=& -m L_i \\
&\dfrac{dv_i}{dt} &=& -\xi v_i \\
&\dfrac{dz_i}{dt} &=& 0
\end{aligned}
\right.
\end{equation}
Following the notation of the lemma \ref{glue_lemma}, we obtain that the solution $(X,Y,\mathbf{L},\mathbf{v},\mathbf{z})$ of $\eqref{CA5}$ with the initial condition $(X,Y,\mathbf{L},\mathbf{v},\mathbf{z})(0)=(X_0,Y_0,\mathbf{L}_0,\mathbf{v}_0,\mathbf{z}_0)\in\R\times \R\times \left(\R^n\right)^3$ verifies
$$\lim_{t\to+\infty}(X,Y,\mathbf{L},\mathbf{v})(t)=\left(X^*(\mathbf{z}_0),Y^*(\mathbf{z}_0),0,0\right).$$
for some functions $X^*(\mathbf{z})$ and $Y^* (\mathbf{z})$  which remains to be compute.\\
The two first equation of \eqref{CA5} reads
\begin{equation} \label{CA5_details1}
\left\{
\begin{aligned}
&\dfrac{dX}{dt} &=& -\beta T^*X 
% \chi_i F_X^i\left(\mathbf{L},\mathbf{K},\mathbf{v},\mathbf{z}\right)
\\
&\dfrac{dY}{dt} &=& \beta(S^* - T^* -k I^*)X - (m+\beta k T^*)Y \\
\end{aligned}
\right.
\end{equation}
%Note  that $\mathbf{v}=0$ implies that the change of variables \eqref{2.14} reads simply 
%$$I_i = I^*z_i,\quad J_i =D^* z_i.$$
So $X^*(\mathbf{z}_0)=0$ and ,$Y^*(\mathbf{z}_0)=0$.
Now, \eqref{eq:Fzi} with the notation of the lemma \ref{glue_lemma}   gives 
\begin{equation} \label{3.31}
\begin{aligned}
f_{z_i}^4(\mathbf{z})=
 \dfrac{1}{|P|}\left[ 2(T^*)^2 I^* \beta k \sum_{j=1}^N (\omega_{ij}^i+\omega_{ji}^i)z_j\right].
\end{aligned}
\end{equation}
Note that $2\beta k {T^*}^2I^* =  2mT^*D^*$. From $\omega_{ji}^i=-\omega_{ji}^j$ we see that
 the slow system \eqref{CAslow_slowpart} reads 
\begin{equation}\label{3.34}
\boxed{
	\dfrac{dz_i}{d\tau} = \Theta_4 z_i\sum_{j=1}^N \left(\omega_{ij}^i-\omega_{ji}^j\right)z_j , \qquad 1 \leq i \leq N
}
\end{equation}
with $\Theta_4=\dfrac{2mT^*D^*}{|P|}$.\\
Denote the $N\times N$ matrix $\Omega=(\omega_{ij}^i)_{i,j}$ and $\Lambda=\Omega-\Omega^T$.  We may rewrite this equation as 
\begin{equation} 
\boxed{\dfrac{dz_i}{d\tau} = \Theta_4z_i \left(\left(\Lambda \mathbf{z}\right)_i\right), \qquad 1 \leq i \leq N.}
\end{equation}
Finally, noting that $\Lambda=-\Lambda^T$ is skew symmetric, we have $\mathbf{z}^T\Lambda \mathbf{z}=0$ so the slow equation reads 
\begin{equation} 
\boxed{
	\dfrac{dz_i}{d\tau} = \Theta_4z_i \left( \left( \Lambda \mathbf{z}\right)_i-\mathbf{z}^T\Lambda \mathbf{z}\right), \qquad 1 \leq i \leq N.}
\end{equation}
which is nothing but \eqref{2.47} with $\lambda^j_i = \omega_{ij}^i-\omega_{ji}^j$.\\ \\
Remark that, this system leads to family of closed trajectories of an odd number $\tilde{N}$ of persistent strains but it is structurally unstable (except if $\tilde{N}=1$), see \cite{Chawanya}. Hence, in this case $\mathcal{A}=\{4\}$ we need to compute the term in $\epsilon^2$ in the expansion, which we do not do in this text.
However, when there are perturbations in other terms then the deviation in this trait conducts to interesting non trivial dynamics, which is shown in sections \ref{sec4.5} and \ref{sec4.6}. This is similar to the case of large $\mu$ with perturbation in co-colonization interaction factor $k_{ij}$, i.e. $\mathcal{A} = \{5\}$, see \cite{Gjini1}. We find that for $\mu \to 0$ and random $\alpha_{ij}$, we have a case of Generalized Lotka-Volterra (GLV) dynamics with constant growth rates and random interactions. Meanwhile, if $\mu \gg 1$, dynamics converge to hyper-tournament dynamics studied by \cite{Allesina5638} for anti-symmetric matrix of interaction $\mathsf{W}$ with $\mathsf{W}_{ij} = \pm 1$; and by \cite{Grilli} for the case in which all the eigenvalues of $\mathsf{W} + \mathsf{W}^T$ are negative.

\subsection{Proof of lemma \ref{lmm_error} of error estimate \label{sec3.5}}
\begin{lmm} \label{lmm11}
	The solution $\left(z_i\right)_{i=1,\dots,N}$ of the slow-fast form system \eqref{CAslow} tends to the solution of the slow system \eqref{CAslow_slowpart} as $\epsilon \to 0$ locally uniformly in time on $[\tau_0,T]$, with $\tau_0 > 0$, $T > \tau_0$ arbitrarily and independent on $\epsilon$.
\end{lmm}
\begin{proof}
	It suffices to verify the conditions for Tikhonov's theorem, see Theorem \eqref{thm2}.\\
	$\large\bullet$ Firstly, we prove that \eqref{CAslow} with initial values possesses the unique solution.\\
	The system \eqref{CAslow} with initial values can be rewritten into
	\begin{equation} \label{3.36}
	\dfrac{dx}{d\tau} = f(x), \qquad x(0) = x_0,
	\end{equation}
	where $x = \left(X,Y,\mathbf{L},\mathbf{v},\mathbf{z}\right)$, then $x(\tau) \in \mathbb{R}^{3N+2}$. We note that the function $f$ of \eqref{3.36} is a vector function with all the components are polynomial of variables $\left(X,Y,\mathbf{L},\mathbf{v},\mathbf{z}\right)$ (explicitly computed in sections \ref{sec2.4}, \ref{sec3.3} and \ref{sec3.4})and we work in the bounded set $[0,T]$ of time where all the functions $\left(X,Y,\mathbf{L},\mathbf{v},\mathbf{z}\right)$ are differentiable. Hence, $f$ is global Lipschitz and the uniqueness of solution for \eqref{CAslow} follows, according to the Picard-Lindelof Theorem, see Theorem 2.2 in \cite{Geral}.\\
	Implement analogously for \eqref{CAslow_slowpart}, we acquire the same conclusion for the uniqueness of solution.\\ \\
	$\large\bullet$ Secondly, by the proof of lemma \ref{lemma_fast}, we have that the solution $(X,Y,\mathbf{L},\mathbf{v},\mathbf{z})$ of $\eqref{CA}$ with any initial condition $$(X,Y,\mathbf{L},\mathbf{v},\mathbf{z})(0)=(X_0,Y_0,\mathbf{L}_0,\mathbf{v}_0,\mathbf{z}_0)\in\R\times \R\times \left(\R^n\right)^3$$
	verifies $\mathbf{z}(t)=\mathbf{z}_0$ for all $t\geq 0$ and 
	$$\lim_{t\to+\infty}(X,Y,\mathbf{L},\mathbf{v})(t)=\mathbf{\Phi}(\mathbf{z}_0)$$
	asymptotically, in which, $\Phi\left(\mathbf{z}\right) = \left(X^*(\mathbf{z}),Y^*(\mathbf{z}),\chi_3\mathbf{L}^*(\mathbf{z}),0\right)$ satisfy the system \eqref{CAslow} in slow timescale, with $\epsilon = 0$ as follows
	\begin{equation}
	\left\{
	\begin{aligned}
	&0 &=& -\beta T^*X +\chi_1 F_X^1\left(\mathbf{v},\mathbf{z}\right)
	+\chi_2 F_X^2\left(\mathbf{v},\mathbf{z}\right)
	+\chi_3 F_X^3\left(\mathbf{L}\right) + O(\epsilon)\\
	&0 &=& \beta(S^* - T^* -k I^*)X - (m+\beta k T^*)Y 
	+\chi_1 F_Y^1\left(\mathbf{v},\mathbf{z}\right)
	+\chi_2 F_Y^2\left(\mathbf{v},\mathbf{z}\right)
	+\chi_5 F_Y^5\left(\mathbf{v},\mathbf{z}\right)\\
	&0 &=& -m L_i + \chi_3 F_{L_i} \left(\mathbf{v},\mathbf{z}\right)\\
%	&0 &=& -mK_i + \chi_5 F_{K_i} \left(\mathbf{v},\mathbf{z}\right) \\
	&0 &=& -\xi v_i
	\end{aligned}
	\right.
	\end{equation}
	Applying Tikhonov's Theorem, we have the required conclusion.
\end{proof}
Let us now approximate the solution of the original dynamics \eqref{2.1} using the solution of slow-fast form \ref{CAslow}, when $\epsilon$ is small enough.
\begin{lmm} \label{thm6}
	Under our assumptions, for any initial values of \eqref{2.1}, there exists $\tau_0 > 0$ and initial value $\mathbf{z}\left(\tau_0\right)$ of \eqref{CAslow}, such that for any $T>\tau_0$, there are $\epsilon_0 > 0$ and $C_T > 0$ satisfies $\forall \epsilon < \epsilon_0$
	\begin{equation} \label{3.37}
	\left|S\left(\dfrac{\tau}{\epsilon} \right) - S^*\right| + \sum_{i=1}^{N}\left|I^*z_i(\tau) - I_i\left(\dfrac{\tau}{\epsilon} \right)\right| + \sum_{i=1}^{N}\left|T^*z_i(\tau) - J_i\left(\dfrac{\tau}{\epsilon} \right)\right| \leq \epsilon C_T ,
	\end{equation}
	for all $\tau_0 \leq \tau \leq T$, where $\left(S, I_i, J_i\right)_{i = 1,\dots,N}$ verifies \eqref{2.1} and $(z_1, \dots, z_N)$ is the solution of \eqref{CAslow}.
\end{lmm}
\begin{proof}
	To prove this lemma, we make two steps, one is to prove the error estimate between $S^*$, $T^*$, $I^*$ and the solution $\left(S,T,I\right)$ of \eqref{2.8}, the other one is approximating the solutions of \eqref{2.8} using the solution of \eqref{CAslow}.\\ \\
	$\large\bullet$ \textit{First step}, we wish to apply the Expansion Theorem \ref{thm5}. Note that, if \eqref{2.8} satisfies the conditions of Theorem \ref{thm5} because of the property of global Lipschitz, then it will also fulfill the conditions of the Picard-Lindelof, see Theorem 2.2 in \cite{Geral}. Thus, if that, for each initial value, \eqref{2.8} always has the unique solution. Therefore, it's suffices to verify the two conditions mentioned in Theorem \ref{thm5}, including the global Lipschitz properties.\\ \\
	Denote $x = \left(S, I_1, I_2, \dots, I_N, J_1, J_2, \dots, J_N \right)$. By the extract of \eqref{2.8} for $S, I_i, J_i$, $1 \leq i \leq N$, we write the system for $\left(S, T, I_i, J_i\right)$, $i = 1,\dots, N$ in \eqref{2.8} into the following form
	\begin{equation}
	\dfrac{dx}{dt} = f_0(t,x) + \epsilon f_1(t,x)
	\end{equation}
	and in any bounded domain $|t-t_0| \leq h$ we have
	\begin{enumerate}
		\item[1.] $f_0(t,x)$ is continuous in $t$, continuously differentiable in $x$;
		\item[2.] $f_1(t,x)$ continuous in $t,x$, Lipschitz-continuous in $x$.
	\end{enumerate}
	According to this extraction, $f_0(t,x)$ and $f_1(t,x)$ are well-defined.  Note that the function $f_0(t,x)$ has the $\left(f_S(t,x), f_{I_1}(t,x), \dots, f_{I_N}(t,x),f_{J_1}(t,x), \dots, f_{J_N}(t,x)\right)$ for $f_S$, $f_{I_i}$, $f_{J_i}$ are functions $\mathbb{R}^{2N+1} \to \mathbb{R}$, for all $1 \leq i \leq N$. The function $f_1(t,x)$ has the same form as well.\\ 
	It's easy to see that $f_0(t,x)$ is continuous in $t$, continuously differentiable in $x$ and the function $f_1(t,x)$ continuous in $t,x$. It remains to prove that $f_1(t,x)$ is Lipschitz-continuous in $x$ in each bounded domain $|t - t_0| \leq h$, for all $h \in \mathbb{R}^+$. Indeed, $f_1(t,x)$ is a polynomial in multi variables $\left(S, T, I_i, J_i \right)$, $i = 1,\dots,N$, and note that $S + T = 1$. In consequence, it is Lipschitz-continuous.\\ \\
	By the earlier arguments, if $x^r = \left(S^r, T^r, I^r_i, J^r_i \right)$ satisfies the neutral system \eqref{2.9} and $x = \left(S, T, I_i, J_i \right)_{1\leq i \leq N}$ satisfying \eqref{2.8} then $\left\|x - x^r\right\|_{\mathbb{R}^{2N+2}} = O(\epsilon)$.\\ \\
	Therefore, note that $I =\sum_{i=1}^{N} I_i$, we deduce the solution of \eqref{2.8} can be approximated using neutral system. Combine with the arguments in section \ref{sec2.3}, the approximation of solution $(S,T,I)$ of \eqref{2.8} by $\left(S^*,T^*,I^*\right)$ is accordingly plausible in the sense of $O\left(\epsilon\right)$. We have done our first step.\\ \\
	$\large\bullet$ \textit{Second step}, we claim that all the algebraic and linear transformations from \eqref{2.8} to \eqref{CAslow} are equivalent with error estimate $O\left(\epsilon\right)$, including changing $\left(S,T,I\right)$ to $\left(X,Y\right)$ using $S^*, T^*, I^*$ (proved in the first part), changing $\begin{pmatrix}
	I_i \\ J_i
	\end{pmatrix}$ to $\begin{pmatrix}
	v_i \\ z_i
	\end{pmatrix}$ (linear operator) and changing to time scale $\tau = \epsilon t$ with re-denote $z\left(\tau\right)$ (see argument in \eqref{CAslow}). We follow the steps of the preceding proof, that are verifying the conditions, and using Expansion Theorem \ref{thm5} once again (note that $v\left(\tau\right) \to 0$ asymptotically), we have that
	\begin{equation*}
	\sum_{i=1}^{N}\left|I^*z_i(\tau) - I_i\left(\dfrac{\tau}{\epsilon} \right)\right| 
	+ \sum_{i=1}^{N}\left|T^*z_i(\tau) - J_i\left(\dfrac{\tau}{\epsilon} \right)\right| = O\left(\epsilon\right),
	\end{equation*}
	for all $\tau_0 \leq \tau \leq T$, where $\left(I_i, J_i\right)_{i=1,\dots,N}$ verify \eqref{2.1} and $(z_1, \dots, z_N)$ is the solution of \eqref{CAslow}.\\ \\
	Combining two parts, we have the conclusion for this lemma.
\end{proof}
By two lemmas \ref{lmm11} and \ref{lmm12}, we have that 
\begin{equation} \label{3.39}
\left|S\left(\dfrac{\tau}{\epsilon} \right) - S^*\right| 
+ \sum_{i=1}^{N}\left|I^*z_i(\tau) - I_i\left(\dfrac{\tau}{\epsilon} \right)\right| + \sum_{i=1}^{N}\left|T^*z_i(\tau) - J_i\left(\dfrac{\tau}{\epsilon} \right)\right| \leq \epsilon C_T ,
\end{equation}
for all $\tau_0 \leq \tau \leq T$, where $\left(S, I_i, J_i\right)_{i = 1,\dots,N}$ verifies \eqref{2.8} and $(z_1, \dots, z_N)$ is the solution of \eqref{CAslow_slowpart}.\\ \\
Finally, we will find an approximation of $I_{ij}$, $1 \leq i \leq N$ and estimate the error. Indeed, according to \eqref{3.39}, we substitute $I_i\left(t\right)$ by $I^*z_i\left(\tau\right) + O(\epsilon)$  and $J_j\left(t\right)$ by $T^*z_j\left(\tau\right)$ in all of the equations for $I_{ij}\left(t\right)$, $1 \leq i,j \leq N$ we have the equations
\begin{equation}
\dfrac{dI_{ij}\left(t\right)}{dt} = \beta_i k_{ij}(I^*z_i\left(\tau\right) + O(\epsilon))(T^*z_j\left(\tau\right) + O(\epsilon)) - m_{ij}I_{ij}\left(t\right), \qquad 1 \leq i,j \leq N,
\end{equation}
which  becomes
\begin{equation}
\dfrac{dI_{ij}\left(t\right)}{dt} = - mI_{ij} + \beta k I^*T^*z_i\left(\tau\right) z_j\left(\tau\right) + O(\epsilon), \qquad 1 \leq i,j \leq N.
\end{equation}
Now we formulate and prove the result for approximations of $I_{ij}$, $1 \leq i,j \leq N$, then deduce the approximation and error estimate for the whole initial system \eqref{2.1}.
\begin{lmm}\label{lmm12}
	Under our assumptions, for any initial values of \eqref{2.1}, there exists $\tau_0 > 0$ and initial value $\mathbf{z}\left(\tau_0\right)$ of \eqref{CAslow_slowpart}, such that for any $T>\tau_0$, there is $\epsilon_0 > 0$ and $C_T > 0$ satisfies $\forall \epsilon < \epsilon_0$
	\begin{equation} \label{3.42}
	 \sum_{i,j = 1}^{N}\left|I_{ij}\left(\dfrac{\tau}{\epsilon}\right)  - k\dfrac{I^*T^*}{S^*}z_i\left(\tau\right)z_j\left(\tau\right)\right| \leq \epsilon C_T ,
	\end{equation}
	for all $\tau_0 \leq \tau \leq T$, $(I_{ij})_{1 \leq i,j \leq N}$ satisfying \eqref{2.1} and $(z_1, \dots, z_N)$ is the solution of reduced system \eqref{CAslow_slowpart}.
\end{lmm}
\begin{proof}
	Assume $\left(I^r_{ij}\right)_{1 \leq i,j \leq N}$ to be the solution of 
	\begin{equation} \label{3.43}
	\dfrac{dI_{ij}\left(t\right)}{dt} = - mI_{ij}\left(t\right) + \beta k I^*T^*z_i\left(\epsilon t\right) z_j\left(\epsilon t\right),
	\end{equation}
	$1 \leq i,j \leq N$. Then, for each $\tau_0>0$ and $T>\tau_0$, after the changing time scale $\tau = \epsilon t$, we have $\sum\limits_{i,j = 1}^{N} \left|I_{ij}\left(\dfrac{\tau}{\epsilon}\right) - I^r_{ij}\left(\dfrac{\tau}{\epsilon}\right) \right| = O(\epsilon)$ for any $\tau \in [\tau_0,T ]$. Indeed, from \eqref{2.1} and \eqref{3.43}, we have that
	\begin{equation}
	\begin{aligned}
	&\dfrac{dI_{ij}}{dt}\left(\dfrac{\tau}{\epsilon} \right) &=& -m_{ij} I_{ij}\left(\dfrac{\tau}{\epsilon} \right) + \beta_j k_{ij} I_i\left(\dfrac{\tau}{\epsilon} \right) J_j\left(\dfrac{\tau}{\epsilon} \right) \\
	&\dfrac{dI^r_{ij}}{dt}\left(\dfrac{\tau}{\epsilon} \right) &=& -m I^r_{ij}\left(\dfrac{\tau}{\epsilon} \right) + \beta k I^* T^* z_i\left(\tau\right) z_j\left(\tau\right)
	\end{aligned}
	\end{equation}
	which implies
	\begin{equation}
	\dfrac{d}{dt}\left(I_{ij}\left(\dfrac{\tau}{\epsilon} \right) - I^r_{ij}\left(\dfrac{\tau}{\epsilon} \right)  \right)= -m\left(I_{ij}\left(\dfrac{\tau}{\epsilon}\right) - I^r_{ij}\left(\dfrac{\tau}{\epsilon} \right) \right) 
	- \epsilon\gamma u_{ij} I_{ij}\left(\dfrac{\tau}{\epsilon}\right)
	+ \left(\beta_jk_{ij}I_i\left(\dfrac{\tau}{\epsilon} \right) J_j\left(\dfrac{\tau}{\epsilon} \right) - \beta k I^* T^* z_i\left(\tau\right) z_j\left(\tau\right)\right).
	\end{equation}
	By lemma \ref{thm6}, we have that $\left|\beta_jk_{ij} I_i\left(\dfrac{s}{\epsilon}\right) J_j\left(\dfrac{s}{\epsilon}\right) - \beta k I^* T^* z_i(s) z_j(s)\right| = O(\epsilon)$ uniformly for $s \in [\tau_0,T]$. It is trivial to note that, since $\left|I_{ij}\right| \leq 1$, $\epsilon\gamma u_{ij} \left|I_{ij}\left(\dfrac{\tau}{\epsilon}\right)\right| = O\left(\epsilon\right)$. 
	Then, for all $1 \leq i,j \leq N$, using the expansion theorem- Theorem \ref{thm5}, we observe that
	\begin{equation}
	\left| I_{ij}\left(\dfrac{\tau}{\epsilon}\right) - I^r_{ij}\left(\dfrac{\tau}{\epsilon}\right) \right| = O(\epsilon).
	\end{equation}
	We then compute the solution $\left(I^r_{ij}\right)_{1\leq i,j \leq N}$ of \eqref{3.43} to be
	\begin{equation}
	I^r_{ij}\left(t\right) = e^{-mt}\left(\beta k I^*T^*\int_{0}^{t}e^{ms}z_i\left(\epsilon s\right)z_j\left(\epsilon s\right)ds + C\right), \qquad C\in \mathbb{R}.
	\end{equation}
	For any fixed time $T$ and $\tau_0 \leq t \leq T$, when $\epsilon \to 0$ we can regard $z_i\left(\epsilon t\right)$ invariant. Hence, for all $1 \leq i,j \leq N$, we have $\left|I^r_{ij}\left(t\right) - k\dfrac{I^*T^*}{S^*}z_i\left(\epsilon t\right)z_j\left(\epsilon t\right)\right| = O\left(\epsilon\right)$, which implies $\left|I_{ij}\left(\dfrac{\tau}{\epsilon} \right) - k\dfrac{I^*T^*}{S^*}z_i\left(\tau\right)z_j\left(\tau\right)\right| = O\left(\epsilon\right)$.
\end{proof}
Combining Lemmas \ref{lmm11},\ref{thm6} and \ref{lmm12}, we have the Lemma \ref{lmm_error}.\\ \\
Thanks to this section, we now have the main result for the error estimate, that allows us to approximate the solution of the original system \eqref{2.1} using the solution of slow system \eqref{2.47}. The original system \eqref{2.1} now formally reduces to the slow system (replicator system) \eqref{2.47}, which leads to many advantages in analysis and prediction. The massive number of equations in \eqref{2.1} now decreases from $N^2 + N + 1$  to $N$ equations of \eqref{2.47}, which helps in computation and time. Thus, we may not need to compute the whole original model \eqref{2.1} to make prediction but only the replicator equations \eqref{2.47}. The main result in section \ref{sec2.5} also has biological meaning, when the coefficients of slow system \eqref{2.47} are pairwise invasion fitness, giving information about survival outcome of 2-strain system as in table \ref{tab:table2}. Furthermore, $\lambda_{i}^j$'s give us the meaning and effects of each trait perturbation on the system and its long time behavior, which can not be seen directly in the \eqref{2.1}.
%%%%%%%%%%%%%%%%%%%%%%%%%%%%%%%%%%%%%%%%%%%%%%%%%%%%%%%%%%%%%%%%%%%%%%%%%%
%%%%%%%%%%%%%%%%%%%%%%%%%%%%%%%%%%%%%%%%%%%%%%%%%%%%%%%%%%%%%%%%%%%%%%%%%%
%%%%%%%%%%%%%%%%%%%%%%%%%%%%%%%%%%%%%%%%%%%%%%%%%%%%%%%%%%%%%%%%%%%%%%%%%%
\section{Illustrations of the model and interpretations}\label{sec5}
%%%%%%%%%%%%%%%%%%%%%%%%%%%%%%%%%%%%%%%%%%%%%%%%%%%%%%%%%%%%%%%%%%%%%%%%%%
%%%%%%%%%%%%%%%%%%%%%%%%%%%%%%%%%%%%%%%%%%%%%%%%%%%%%%%%%%%%%%%%%%%%%%%%%%
%%%%%%%%%%%%%%%%%%%%%%%%%%%%%%%%%%%%%%%%%%%%%%%%%%%%%%%%%%%%%%%%%%%%%%%%%%

In this section, we present some results and simulations about survival outcome of model based on the replicator system \eqref{2.47}. Initially, we recall the definition of basic reproduction ratio of strain $i$ that is the expected number of secondary cases produced by a single (typical) infection of strain $i$ in a completely susceptible population and computed by $R_{0,i} = \dfrac{\beta_i}{m_i}$. If there is only variation in transmission rates among strains, then $R_{0,i}$'s fully determine the unique winner in the system. Yet, in cases of variation in transmission and clearance rates, it can be shown that $R_{0,i}$'s alone do not determine the survival outcome anymore because of the feedbacks induced by persistence in the coinfection compartment. These phenomena are illustrated in proofs and numerical simulations as follows in this section.
\subsection{Competitive exclusion due to variation in transmission and infection clearance rates only, $\mathcal{A}=\{1\}$, $\mathcal{A} = \{2\}$, and $\mathcal{A} = \{1,2\}$}\label{sec4.1}
\subsubsection{Variation in transmission rate or infection clearance rates $\mathcal{A}=\{1\}$ or $\mathcal{A} = \{2\}$}\label{sec:A1}
Now, we show the competitive exclusion principle in this case $C_{\mathcal{A}}$ with $\mathcal{A}=\{1\}$. In these cases, the competitive exclusion principle holds: the species with the largest $R_{0,i}$ is the only survivor.
\begin{thm} \label{thm9}
	Assume that $\mathcal{A} = \{1\}$ and $b_1 > b_2 \geq \dots \geq b_N$. Then $E_1 = (1,0,\dots,0)$ is globally stable in $(0,1) \times [0,1]^{N-1} \cap \{u \in \mathbb{R}^N: \sum_{i=1}^{N}z_i =1 \}$.
\end{thm}
This result means that, the strain with the largest basic reproduction number is the best competitor. However, in general, this fact does not always occurs, which we will illustrate in later subsection.
\begin{proof}
	For simplicity, denote $D = (0,1) \times [0,1]^{n-1} \cap \{u \in \mathbb{R}^N: \sum_{i=1}^{N}z_i =1 \}$.\\	
	We aim to use LaSalle's invariant principle. Consider $V(u) = -\ln z_1$. Since we are consider the coexistence in $D$, then
	\begin{equation}
	\begin{aligned}
	\dfrac{dV(u)}{d\tau} &= -\Theta_1\left(b_1 - \sum_{j=1}^{N}b_jz_j \right) = -\Theta_1\left(b_1 \sum\limits_{j=1}^{N}z_j -\sum\limits_{j=1}^{N}b_jz_j \right) = -\Theta_1\sum\limits_{j=1}^{N}(b_1 - b_j)z_j.
	\end{aligned}
	\end{equation}
	It's straightforward that $V(u) > 0$ because $0 < z_1 < 1$ in $D$. Because of the assumption $b_1 = \max\{b_i; 1 \leq i \leq N\}$ then $b_1 - b_j$ must be positive for all $j\neq 1$. Recall that $\Theta_1 > 0$ then, $\dfrac{dV(u)}{d\tau} \leq 0$. We have that
	\begin{equation}
	\dfrac{dV(u)}{d\tau} = 0 \Leftrightarrow (b_1 - b_j)z_j = 0, \quad \forall j \quad \Leftrightarrow
	\left\{
	\begin{aligned}
	&z_j &=& 1, \quad &j=1&\\
	&z_j &=& 0, \quad &j \neq 1&.
	\end{aligned}
	\right.
	\end{equation}
	Thus, $V(u)$ is a Lyapunov function associated to $u\left(\tau\right)$. Applying LaSalle's invariant principle, we obtain our solution $u$ tends to $E_1$ asymptotically.
\end{proof}
Analogously, we have a similar result for $\mathcal{A} = \{2\}$, that states that, the strain with smallest single infection clearance rate (longest duration of carriage) is the unique survivor.
\begin{thm}
	Assume that $\mathcal{A} = \{2\}$ and $\nu_1 < \nu_2 \leq \dots \leq \nu_N$. Then $E_1 = (1,0,\dots,0)$ is globally stable in $(0,1) \times [0,1]^{N-1} \cap \{u \in \mathbb{R}^N: \sum_{i=1}^{N}z_i =1 \}$.
\end{thm}
The proof for this result uses the same argument in the theorem \ref{thm9} so we do not present it.
%%%%%%%%%%%%%%%%%%%%%%%%%%%%%%%%%%%%%%%%%%%%%%%%%%%%%%%%%%%%%%%%%%%%%%%%%%%%%%%%%%%%%%%%%%%%%%%%%%%%%%%%%%%%%%%%%%%%%
%%%%%%%%%%%%%%%%%%%%%%%%%%%%%%%%%%%%%%%%%%%%%%%%%%%%%%%%%%%%%%%%%%%%%%%%%%%%%%%%%%%%%%%%%%%%%%%%%%%%%%%%%%%%%%%%%%%%%
\subsubsection{Variation in transmission and single infection clearance rates, $\mathcal{A}=\{1,2\}$}\label{sec4.2}
%%%%%%%%%%%%%%%%%%%%%%%%%%%%%%%%%%%%%%%%%%%%%%%%%%%%%%%%%%%%%%%%%%%%%%%%%%%%%%%%%%%%%%%%%%%%%%%%%%%%%%%%%%%%%%%%%%%%%
%%%%%%%%%%%%%%%%%%%%%%%%%%%%%%%%%%%%%%%%%%%%%%%%%%%%%%%%%%%%%%%%%%%%%%%%%%%%%%%%%%%%%%%%%%%%%%%%%%%%%%%%%%%%%%%%%%%%%
In this subsection, it is shown that $R_{0,i}$'s do not determine the unique survivor anymore when $\mathcal{A} = \{1,2\}$ by constructing a counterexample. Firstly, we need an auxiliary lemma. With system $C_{\mathcal{A}}$ with $\mathcal{A}=\{1,2\}$, we try to make a result similar to Theorem \ref{thm9} about the longtime scenarios of competition for 
\begin{equation} \label{4.3}
\left\{
\begin{aligned}
&\dfrac{dz_i}{d\tau} = \Theta_1z_i\left(b_i - \sum_{i=1}^{N}b_jz_j\right) + \Theta_2z_i\left(- \nu_i + \sum_{i=1}^{N}\nu_jz_j\right)\\
&z_1 + z_2 + \dots + z_N = 1.
\end{aligned}
\right.
\end{equation}
Recalling that  $\Theta_1,\Theta_2> 0$ by definitions, we can prove the following theorem stating that the competitive exclusion occurs again but depends on the parameters of the neutral model though  the quantity $\Theta_1b_j - \Theta_2\nu_j$ which characterizes the unique survivor.
Note that this 
\begin{thm} \label{thm11}
	Assume in \eqref{4.3} with $N$ strains, there exists a strain, namely $1$, satisfies $\Theta_1b_1 - \Theta_2\nu_1 = \max\limits_{1 \leq j \leq N}\{\Theta_1b_j - \Theta_2\nu_j \}$. Then $E_1 = (1,0,\dots,0)$ is globally stable in $D = (0,1) \times [0,1]^{N-1} \cap \{u \in \mathbb{R}^N : \sum_{i=1}^{N}z_i =1 \}$.
\end{thm}
\begin{proof}
	Analogously to the earlier result in section \ref{sec4.1}, we want to apply LaSalle's invariant principle. Consider the function $V(u) = -\ln z_1$ then by our hypothesis, it's easy to see that $V(u) > 0$ and 
	\begin{equation}
	\begin{aligned}
	\dfrac{dV(u)}{d\tau} = -\sum_{j=1}^{N}\left(\Theta_1b_1 - \Theta_2\nu_1 -\Theta_1b_j +\Theta_2\nu_j\right)z_j
	 \leq 0.
	\end{aligned}
	\end{equation}
	Hence, $V(u)$ is an association Lyapunov function. The equation $\dfrac{dV(u)}{d\tau} = 0$ is equivalent to
	\begin{equation}
	\left\{
	\begin{aligned}
	&\left[\left(\Theta_1b_i - \Theta_2\nu_i\right) - \left(\Theta_1b_j - \Theta_2\nu_j\right)\right]z_j  &=& 0, \qquad 1 \leq j \leq N \\
	&z_1 + z_2 + \dots + z_N &=& 0
	\end{aligned}
	\right.
	\end{equation}
	which is equivalent to $(z_1,z_2,\dots,z_N) = (1,0,\dots,0)$. By LaSalle's invariant, $E_1$ is globally stable in $D$.
\end{proof}
We next come to see how this result is used in the forthcoming examples. We then compare the results with relations of $R_{0,i}$ to see how basic reproduction numbers affect the final competitive outcomes. Firstly, with the perturbations existing in clearance rates, the $R_{0,i}$ now becomes 
\begin{equation}
R_{0,i} = \dfrac{\beta_i}{m_i} = \dfrac{\beta + \epsilon  b_i}{m + \epsilon  \nu_i} = \dfrac{\beta}{m}(1 + \dfrac{\epsilon}{\beta} b_i)\left(1 -  \dfrac{\epsilon}{m}\nu_i\right) + O(\epsilon^2),
\end{equation}
which is equivalent to $R_{0,i} = \dfrac{\beta}{m} + \epsilon \dfrac{\beta}{m}\left(\dfrac{b_i}{\beta} - \dfrac{\nu_i}{m} \right) +O(\epsilon^2)$.\\ 
Hence, note that $R_0 = \dfrac{\beta}{m}$ we have that $R_{0,i} \leq R_{0,j}$ if and only if $b_i- b_j \leq R_0\left(\nu_i - \nu_j\right)$ when $\epsilon \to 0$.
\begin{exmp} \label{exA2}
%	In this case, the system \eqref{4.3} turns to
%	\begin{equation} \label{4.19}
%	\dfrac{dz_i}{d\tau} = z_i \left[\Theta_1\left(b_i -\sum_{j=1}^{N}b_j z_j \right) + \Theta_2\left(-\nu_i + \sum_{j=1}^{N}\nu_j z_j\right)\right].
%	\end{equation}
%	We have two approaches. The first one is that imitating the argument in Theorem \ref{thm9} by writing \eqref{4.19} into
%	\begin{equation}
%	\dfrac{dz_i}{d\tau} = z_i \left[\left(\Theta_1b_i - \Theta_2\nu_i\right) -  \sum_{j=1}^{N}\left(\Theta_1b_j - \Theta_2\nu_j\right) z_j\right].
%	\end{equation}
Consider the system \eqref{4.3}.\\
	Initially, we can directly apply Lemma \ref{thm11} and infer that the strain, called $1$, satisfying $\Theta_1b_1 - \Theta_2\nu_1 = \max\limits_{1 \leq j \leq N}\{\Theta_1b_j - \Theta_2\nu_j \}$ will be the winner.\\ \\
	Yet, unlikely to such result in section \ref{sec4.1}, according to the explicit calculation on $R_{0,i}$, we can construct so that this strain 1 may not have the biggest basic reproduction number. Indeed, $\Theta_1b_i - \Theta_2\nu_i \geq \Theta_1b_j - \Theta_2\nu_j $ is equivalent to
	\begin{equation}
	b_i - b_j  \geq  \dfrac{1}{2}\dfrac{\gamma}{\beta}R_0\dfrac{1}{1+ k\left(R_0 - 1\right)}\left(\dfrac{1}{1+ k\left(R_0 - 1\right)} +1 \right)\left(\nu_i - \nu_j\right).
	\end{equation}
	We can choose $b_i$, $\nu_i$, $1 \leq i \leq n$ and $\gamma>0$, $r>0$, $k>0$ and $R_0>1$  such that for $j \neq 1$,
	\begin{equation}
	\left\{
	\begin{aligned}
	&b_1 - b_j  \geq  \dfrac{1}{2}\dfrac{\gamma}{\beta}R_0\dfrac{1}{1+ k\left(R_0 - 1\right)}\left(\dfrac{1}{1+ k\left(R_0 - 1\right)} +1 \right)\left(\nu_1 - \nu_j\right),\\
	&b_1 - b_j  \leq  R_0 \left(\nu_1 - \nu_j\right),
	\end{aligned}
	\right.
	\end{equation}
	then strain $1$ has  $\Theta_1b_1 - \Theta_2\nu_1 = \max\limits_{1 \leq j \leq N}\{\Theta_1b_j - \Theta_2\nu_j \}$ and $R_{0,1} = \min\limits_{1\leq i\leq N}\{R_{0,i}\}$.\\
	It is possible because $$ \dfrac{1}{2}\dfrac{\gamma}{\beta}R_0\dfrac{1}{1+ k\left(R_0 - 1\right)}\left(\dfrac{1}{1+ k\left(R_0 - 1\right)} +1 \right)  < R_0$$ 
	and we can pick, for instance, $\nu_1 = \max\limits_{1 \leq j \leq N}\{\nu_j\}$, then easily find satisfactory $b_i,\nu_i$.\\
	This example shows us that, even a strain $i$ with smallest basic reproduction number $R_{0,i}$ can be the single competitively exclusive strain if there is variation in both transmission and  clearance rates in a system with co-infection. Explicitly, the strain $1$ is the only survivor but it has the smallest $R_0$.
\end{exmp}
Hence, we can see that, even when there is competitive exclusion, $R_{0,i}$ alone still do not determine the winner if there are perturbations in the transmission rates and clearance rates. More detailed consideration of such effects and interplay between parameters for the 2-strain general system is provided in \cite{thaoEridaSten}. To close this subsection, we present simulations in figure \ref{exclusion} of competitive exclusion to illustrate claims in sections \ref{sec:A1} and \ref{sec4.2}. We choose the 10-strain system and plot frequencies of strains in two cases: perturbation in only transmission rates $\beta_i$; and, perturbation in transmission rates $\beta_i$ and in clearance rates of single colonization $\gamma_{i}$.
\begin{figure}[htb!]
	\centering
	\includegraphics[width= 0.9 \textwidth]{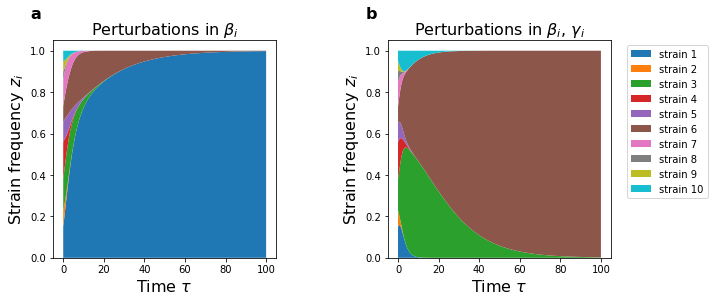}
	\caption{\textbf{Illustration of competitive exclusion dynamics for $N=10$ when strains vary in transmission and clearance rates.} We choose the parameter values of the neutral system  $\beta = 4$, $m = 2$, $\gamma = 1$ and $k = 1.5$. 
The variation of $\beta$ is given by $b = \protect\begin{pmatrix} b_1 & b_2 & \dots & b_N \protect\end{pmatrix}$ and is  
set to be the equal in both cases and equals $b  = \protect\begin{pmatrix} 0.25 & -0.2 & 0.125 & -0.125 & 0.075 & 0.225 & 0.05 & -0.5 & -0.175 & 0 \protect\end{pmatrix}.$ The matrix of $\nu_i$ in \textbf{(b)} is chosen to be $\nu = \protect\begin{pmatrix} 1 & 0.8 & -1.5 & -0.5& 0.3 & -1 & 1.2 & -2 & 0.7 & -2 \protect\end{pmatrix}.$  \textbf{(a)} Strains vary only in transmission rates $\beta_i$: $\mathcal{A}=\{1\}$. \textbf{(b)} Strains vary in transmission and clearance rates $\beta_i, \gamma_i$: $\mathcal{A} = \{1,2\}$. We can see that competitive exclusion is the only outcome in either case. However in \textbf{(a)} the strain with the highest reproduction number will persist while all other strains will go extinct. In contrast, in \textbf{(b)} the coinfection parameters matter, and it is not true that the strain with highest $R_0$ will persist. In this example strain 10 has highest basic reproduction number but strain 6 is the ultimate winner, because of its exact advantage in clearance rate (as explained in Example \ref{exA2}).(\href{https://github.com/lthminhthao/quasineutral_coinfection_N-strains_multiple-variations/blob/master/1_competitive_exclusion.ipynb}{Data \& Codes })}\label{exclusion}
\end{figure}
%%%%%%%%%%%%%%%%%%%%%%%%%%%%%%%%%%%%%%%%%%%%%%%%%%%%%%%%%%%%%%%%%%%%%%%%%%%%%%%%%%%%%%%%%%%%%%%%%%%%%%%%%%%%%%%%%%%%%
%%%%%%%%%%%%%%%%%%%%%%%%%%%%%%%%%%%%%%%%%%%%%%%%%%%%%%%%%%%%%%%%%%%%%%%%%%%%%%%%%%%%%%%%%%%%%%%%%%%%%%%%%%%%%%%%%%%%%
\subsection{Variation in clearance rate of co-colonization may yield  coexistence}\label{sec4.3}
%%%%%%%%%%%%%%%%%%%%%%%%%%%%%%%%%%%%%%%%%%%%%%%%%%%%%%%%%%%%%%%%%%%%%%%%%%%%%%%%%%%%%%%%%%%%%%%%%%%%%%%%%%%%%%%%%%%%%
%%%%%%%%%%%%%%%%%%%%%%%%%%%%%%%%%%%%%%%%%%%%%%%%%%%%%%%%%%%%%%%%%%%%%%%%%%%%%%%%%%%%%%%%%%%%%%%%%%%%%%%%%%%%%%%%%%%%%
\subsubsection{Variation in clearance rate of co-colonization only, $\mathcal{A}=\{3\}$}
\label{sec:A3}
%%%%%%%%%%%%%%%%%%%%%%%%%%%%%%%%%%%%%%%%%%%%%%%%%%%%%%%%%%%%%%%%%%%%%%%%%%%%%%%%%%%%%%%%%%%%%%%%%%%%%%%%%%%%%%%%%%%%%
In this case, the very first claim about competitive outcomes of the system is that, in contrast to the above cases 
$\mathcal{A}\subset\{1,2\}$, there can be coexistence of strains. Indeed, in this case the system can be rewritten on the form of a replicator system with  a symmetric matrices for which several results exists (see in particular \cite{Hofbauer}). In particular we have :

\begin{thm}
Let $\mathcal{A}=\{3\}$ which means variation in coinfection clearance rates only. The system \eqref{2.47} may be rewritten as 
$$\begin{cases} \
\dot{z_i}=2\Theta_3\left(\left(\left(-\bar{U}\right) z\right)_i-z^T \left(-\bar{U}\right)z \right),\quad 1\leq i\leq N\\
z_1+\cdots+z_n=1.
\end{cases}$$
where the symmetric matrix $\bar{U}=\frac12\left(U+U^T\right)$ is symmetric part of the matrix of perturbation $U=\left(u_{ij}\right)_{1\leq i,j \leq N}$.\\
In particular, the function $z\mapsto z^T \bar{U}z$ is a strict Lyapunov function and any positive asymptotic equilibria $z^*$ are globally stable. 
\end{thm}
\begin{proof}
We refer here to the theorem 7.8.1 page 82 of \cite{Hofbauer} for the results about a replicator system with a symmetric matrix $A$.
Then we only have to prove that the system \eqref{2.47} may be rewritten in terms of the symmetric matrix $-\bar{U}$.\\
This comes from the following general fact in the  replicator equation. 
Let $x=(x_j)_{1\leq j\leq N}$ be a vector  and  $A=(a_{ij})_{1\leq i,j\leq N}$ and $C=(c_{ij})_{1\leq i,j\leq N}$ be two $N\times N$ matrix such that $c_{ij}=a_{ij}+x_{j}$.
For every $z=(z_1,\cdots,z_n)$ we have
 $$(C z)_k-z^T Cz 
 = \sum_{j=1}^N c_{kj}z_j- \sum_{i,j}c_{ij}z_jz_i 
 = \sum_{j=1}^N a_{kj}z_j+\sum_{j=1}^N x_{j}z_j- \sum_{i,j}a_{ij}z_jz_i- \sum_{i,j}x_{j}z_jz_i,\quad 1\leq k\leq N.$$
If  $\sum_{i=1}^N z_i=1$ then $\sum_{i,j}x_jz_jz_i=\sum_{j=1}^N x_{j}z_j$ which yields

$$(C z)_k-z^T Cz=(Az)_k-z^TAz,\quad 1\leq k\leq N.$$
The proof follows from the explicit expression of \eqref{2.47} when $\mathcal{A}=\{3\}$ and by taking
$a_{ij}=\frac12 \lambda_i^j= u_{jj}-\frac12\left(u_{ij}+u_{ji}\right)$, $x_j=-u_{jj}$ and $c_{ij}= -\frac12\left(u_{ij}+u_{ji}\right)$.
\end{proof}

\begin{figure}[htb!]
	\centering
	\captionsetup{singlelinecheck=off}
	\includegraphics[width = 1.1 \textwidth]{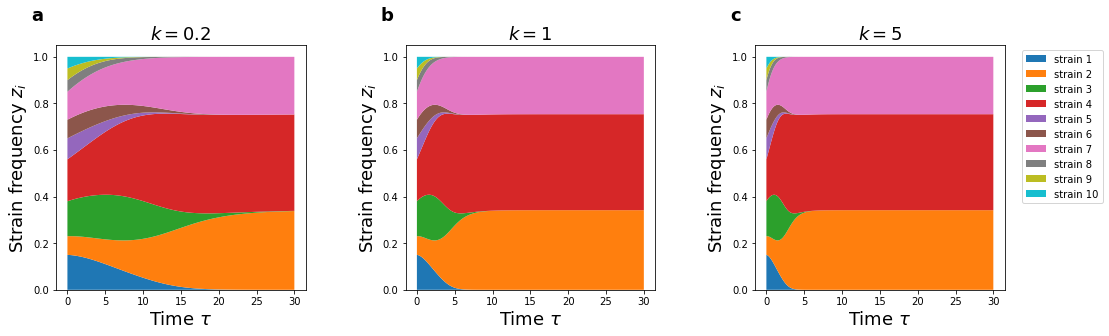}
	\caption[]{\textbf{Strain coexistence is possible when there is variation in coinfection clearance rate and the speed of the dynamics depends on the parameters of the neutral model.} 
		%We choose neutral transmission rate $\beta = 4$ and $\gamma = 1$, basic reproduction ratio $R_0 = 2$
%		and the values of $ u_{ij}$'s are given by the following matrix, i.e. $ u_{ij}$ is at intersection of row $i$ and column $j$,
%		\begin{equation}\label{coin_clearance_matrix}
%		\protect\begin{pmatrix} 
%		3 &  3 & -3 & -5 & -4 & -1 &  2 & -2 & -3 &  3 \\
%		2 & -1 &  1 & -4 & -4 & -1 & -3 &  3 &  4 &  0 \\
%		-4 & -4 & -3 & -3 &  4 & -4 &  3 & -1 &  3 & -2\\
%		-3 & -5 &  0 & -1 &  0 &  1 & -5 & -2 &  4 &  3\\
%		-4 &  2 &  0 &  2 & -2 & -2 & -4 & -2 &  0 & -2\\
%		1 & -5 & -3 &  2 &  3 & -2 & -1 &  4 & -3 & -1\\
%		2 & -5 & -5 & -5 & -4 & -4 &  1 & -3 & -1 & -2\\
%		-4 &  1 & -1 & -1 &  4 & -5 &  3 & -1 &  3 & -1\\
%		-5 &  0 &  2 &  3 &  4 & -3 & -5 & -4 &  2 &  2\\
%		-3 &  2 &  3 &  3 &  1 & -1 &  1 &  2 & -3 & -2
%		\protect\end{pmatrix}.
%		\end{equation}
%	In addition, the initial values are picked to be
%	\begin{equation} \label{initial_value}
%	\begin{pmatrix} 0.15 & 0.08 & 0.15 & 0.18 & 0.09 & 0.08 & 0.12 & 0.05 & 0.05 & 0.05 \end{pmatrix}.
%	\end{equation}
	Here, we illustrate coexistence dynamics under the effect of $k$ for $k = 0.2$ \textbf{(a)}, $k = 1$ \textbf{(b)} and $k = 5$ \textbf{(c)}. In the top sub-panels we show the dynamics of 10 strain frequencies. We choose $\beta = 4$, and basic reproduction number $R_0 = 2$. It can be seen that as $k$ increases, the system tends to its stable state faster. In figures \textbf{(a, b, c)}, three strains 2, 4, 7 coexist after a long time. (\href{https://github.com/lthminhthao/quasineutral_coinfection_N-strains_multiple-variations/blob/master/2_variation_in_coinfection-clearance-rates.ipynb}{Data \& Codes})}\label{A3}
\end{figure}

Two important features of the dynamics in the case $\mathcal{A}=\{3\}$ are: 
\begin{itemize}
\item Large possibilities of stable coexistence steady states.
\item The parameters of the neutral models affect only the speed of the dynamics, given by $\Theta_3$, but not the qualitative behavior. The latter depends only on the symmetric part of the perturbation $U=\left(u_{ij}\right)$.
\end{itemize}

 For an illustration of this case, we take the following example.\\ 
We consider a system of $N=10$ strains with $\mathcal{A}=\left\{3\right\}$. 
In figure \ref{A3}, we plot strains frequencies for multiple values of $k$ showing that the {\it same} coexistence equilibrium of $3$ strains is achieved with a speed dependent on $k$. Note that a similar effect would hold if we vary $R_0$.

We note that the speed of the dynamics is given by
$$\Theta_3 = \dfrac{\gamma T^*D^*}{|P|} = \dfrac{\gamma T^*}{2T^* + \frac{S^*}{k}\left(1 + \frac{m}{m + \beta k T^*}\right)},$$
which increases with  $k$. 
%Thus, changing $k$ only varies absolute value of $\lambda^j_i$ without altering its sign. This fact explains the bolder/lighter phenomenon of matrices $\Lambda$'s in figure \ref{A3}. 
Thus, in this case, increasing $k$  only multiplies whole matrix $\Lambda$ by a factor, which increases  the speed of the convergence to the stable state of coexistence.

%%%%%%%%%%%%%%%%%%%%%%%%%%%%%%%%%%%%%%%%%%%%%%%%%%%%%%%%%%%%%%%%%%%%%%%%%%%%%%%%%%%%%%%%%%%%%%%%%%
%%%%%%%%%%%%%%%%%%%%%%%%%%%%%%%%%%%%%%%%%%%%%%%%%%%%%%%%%%%%%%%%%%%%%%%%%%%%%%%%%%%%%%%%%%%%%%%%%%
\subsubsection{Variation in transmission and coinfection clearance rates, $\mathcal{A}=\{1,3\}$}
\label{sec:A13}%%
%%%%%%%%%%%%%%%%%%%%%%%%%%%%%%%%%%%%%%%%%%%%%%%%%%%%%%%%%%%%%%%%%%%%%%%%%%%%%%%%%%%%%%%%%%%%%%%%%%
%%%%%%%%%%%%%%%%%%%%%%%%%%%%%%%%%%%%%%%%%%%%%%%%%%%%%%%%%%%%%%%%%%%%%%%%%%%%%%%%%%%%%%%%%%%%%%%%%%
When $\mathcal{A}\subset\{1,2,3\}$, the perturbations occur both on traits $\{1\}$ and $\{2\}$ leading on competitive exclusions and on $\{3\}$ leading on coexistence. Thus the relative weights of the perturbation,  depending on the parameters on the neutral model, will affect the {\em qualitative} outcomes of the  dynamics among strains.\\

Hence, unlike in section \ref{sec:A3}, the qualitative behavior  does not depend only on the pertubations $b_i$, $\nu_i$ and $u_{ij}$ but also on the values of the parameters of the neutral model. A precise generic result is out of the scope of this paper. \\
For simplicity, consider the case $\mathcal{A}=\{1,3\}$.
From the formula $\lambda_i^j = \Theta_1 (b_i - b_j) + \Theta_3 (- u_{ij} -  u_{ji} + 2 u_{jj})$, we infer that  the larger the ratio $\frac{\Theta_3}{\Theta_1}$ is the more chance a coexistence may happen. From $$\dfrac{\Theta_3}{\Theta_1} = \dfrac{R_0}{2}\dfrac{\gamma}{\frac{m}{k} + \beta T^*},$$
we see in particular that this ratio increases with $k$. This  is illustrated in the figure \ref{A4}.
 We  see in this figure that shifting $k$ alters qualitatively the dynamics and the ultimate outcome among  strains. In figure \ref{A4}\textbf{(a)} $k = 0.1$ and  the only winner is strain 8, whereas for $k=1$, figure \ref{A4}\textbf{(b)}, then the winners turn to strain 3 and 6. Finally, for $k=3$, the outcome in figure \ref{A4} \textbf{(c)} is the coexistence of strains 2, 4 and 7.\\
Note that the short explanation above, gives only an overview of the phenomena and do not explain  all the details. For instance, we observe that the set of coexistent species depends on the value of $k$ in a complex maner.
%Next, we give another numerical example to see the final outcome of a system in which changing only the paramters of the neutral system can alter the {\em qualitative} outcomes of the  dynamics among strains.\\
 %Figures \ref{A4} \textbf{(d, e, f)} is plotted for the intuitive view of pairwise invasion fitness matrix.
\begin{figure}[htb!]
	\centering
	\captionsetup{singlelinecheck=off}
	\includegraphics[width = 1.1 \textwidth]{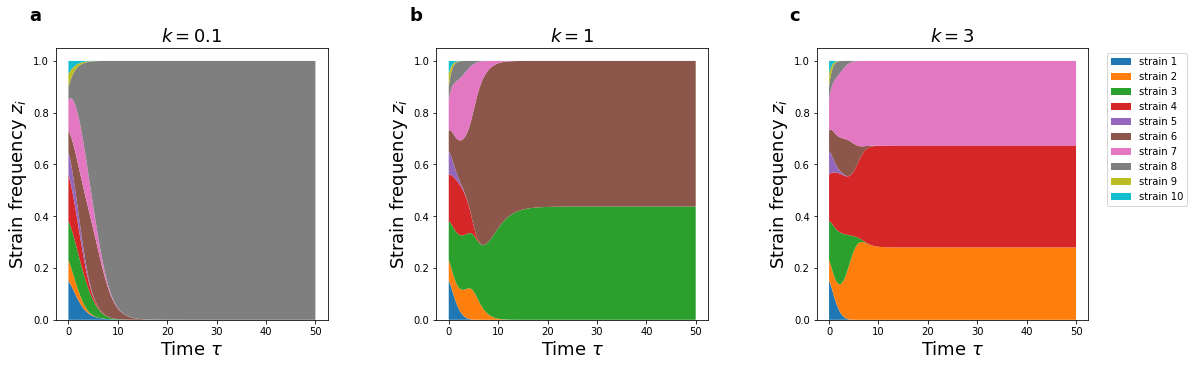}
	\caption{\textbf{The final ecological outcome can shift with changing vulnerability to coinfection, when strains vary in transmission and coinfection clearance rates.} We illustrate coexistence dynamics for $k = 0.1$ \textbf{(a)}, $k = 1$ \textbf{(b)}, $k = 3$ \textbf{(c)}. In the top sub-panels we show the dynamics of 10 strain frequencies. We choose $\beta = 4$, $R_0 = 5$ and $\gamma = 0.5$. %and use the matrix of the values $ u_{ij}$'s to be twice of $ u_{ij}$'s matrix in \eqref{coin_clearance_matrix}.
	We keep the initial values in \ref{A3} and the matrix of value's $b_i$ as follows, in which $b_i$ is in cell $i$-th $\protect\begin{pmatrix}
	0 & -0.2 & 0.125 & -0.125 & 0.225 & 0.75 & 0.5 & 1.25 & -0.175 & 0
	\protect\end{pmatrix}.$ We plot for multiple values of $k$ respectively equal to $0.1$, $1$ and $3$, to show effects of $k$ to transient phenomena. It can be seen that as $k$ increases, changes the survival strains. (\href{https://github.com/lthminhthao/quasineutral_coinfection_N-strains_multiple-variations/blob/master/3_variation_in_coinfection-clearance-rates_transmission-rates.ipynb}{Data \& Codes})}
	\label{A4}
\end{figure}
%In figure \ref{A4}, we take only perturbation in specific transmission rates and co-infection clearance rates, i.e. $\mathcal{A}=\left\{1,3\right\}$.

%Unlike section \ref{sec4.3}, changing co-colonization interaction factor $k$ in this case $\mathcal{A} = \{1,3\}$ shifts the set of survival strains. It is hard to verify by normal eyes that, varying $k$ flips the color of some cells (between red and blue), i.e. changing sign of some $\lambda^j_i$'s, in figure \ref{A4} \textbf{(d, e, f)}. Actually, by practical computation, we have the sign turning occurs in some $\lambda^j_i$'s with $\left(i,j\right)$'s are $(1,8)$, $(2,1)$, $(2,6)$, $(3,6)$, $(3,9)$, $(6,10)$, $(7,6)$, $(8,2)$, $(8,3)$, $(8,10)$ and $(9,6)$. Thus, the changing in survival strain set makes sense. On the other hand, this possibility of sign change (but not much) is because $\lambda^j_i = \Theta_1 (b_i - b_j) + \Theta_3 (- u_{ij} -  u_{ji} + 2 u_{jj})$, increasing $k$  implies that the absolute value of $\Theta_3 (- u_{ij} -  u_{ji} + 2 u_{jj})$ becomes larger. Consequently, in some cases, the changes in absolute values of these terms can change the sign of invasion fitness.
%%%%%%%%%%%%%%%%%%%%%%%%%%%%%%%%%%%%%%%%%%%%%%%%%%%%%%%%%%%%%%%%%%%%%%%%%%%%%%%%%%%
\subsection{Variation in transmission probability from mixed carriage may lead to cycles among  strains.}\label{sec4.5}
%%%%%%%%%%%%%%%%%%%%%%%%%%%%%%%%%%%%%%%%%%%%%%%%%%%%%%%%%%%%%%%%%%%%%%%%%%%%%%%%%%%
In this subsections, we make simulations in which variation at least in transmission probability from mixed carriage, $4 \in \mathcal{A}$. Despite of the antisymmetric matrix of pairwise invasion fitness $\Lambda = \left(\lambda^j_i\right)_{i,j}$ as in cases $\mathcal{A} = \{1\}$ and $\mathcal{A} = \{2\}$, there are many long time behaviors that may occur in this case. In \cite{Chawanya}, one proves that there can be coexistence with higher possibility than competitive exclusive of one strain. However, when there are combinations with other trait perturbation, the outcome survival can shift due to neutral parameters, which will be presented in the next subsections \ref{sec:A124} and \ref{sec4.6}.
%%%%%%%%%%%%%%%%%%%%%%
\subsubsection{Variation in transmission rates and transmission probability from mixed carriage,  $\mathcal{A} = \{1,2,4\}$}
\label{sec:A124}
%%%%%%%%%%%%%%%%%%%%%%%
We make simulations when perturbations in transmission rates $\beta_i$ and transmission capacity of a strain by a host co-colonized. From \eqref{2.47} when $\mathcal{A}\subset\{1,2,4\}$, the equations for this case can be written as
\begin{equation*}
\dfrac{dz}{d\tau} = z\cdot\left(\Lambda z\right)
\end{equation*}
where anti-symmetric matrix $\Lambda$ is the invasion fitness matrix with 
$$\Lambda_{i}^j = \Theta_1\left(b_i - b_j\right) +\Theta_2\left(\nu_j-\nu_i\right)+ \Theta_4\left(\omega^i_{ij} - \omega^j_{ji}\right).$$ 

this type of replicator equation is known as a zero sum games tournaments from which several results are known (see \cite{Chawanya}). In particular the set $E$ of persistent strains is unique, regardless the initial values, and the number of persistent strains is odd. 
\begin{itemize}
\item If this odd number is 1, then the  competitive exclusion principle occurs, as we saw above in the particular case $\emptyset\neq \mathcal{A}\subset \{1,2\}$.
\item If this odd number is above 1, the system is conservative and structurally unstable: there is a family of cycles around a single steady states of these strains $E$. This is possible thus the effect of a perturbation in $\omega^i_{ij}$ (i.e. $4\in\mathcal{A})$.\\
\end{itemize}

As in the section \ref{sec:A13}, the parameters of the neutral models affect the relative weight of the pertubation and therefore the type of outcomes.\\

In figure \ref{A5}, we take $\mathcal{A}=\{1,4\}$. We have   $$\dfrac{\Theta_4}{\Theta_1} =  \dfrac{k\left(R_0 - 1\right)}{1 + k\left(R_0 - 1\right)} = \dfrac{1}{\mu + 1}.$$
Hence, changing $\mu = \dfrac{1}{k\left(R_0 - 1\right)}$ shift the dynamics outcome.
 When $\mu = 0.6$, i.e. small enough, makes $\dfrac{\Theta_1}{\Theta_4}$ large yielding to a cycle  of 3 persistent strains. Conversely, $\mu = 1.2$  large enough leads to the competitive exclusion.
\begin{figure}[htb!]
	\centering
	\captionsetup{singlelinecheck=off}
	\includegraphics[width = 0.8 \textwidth]{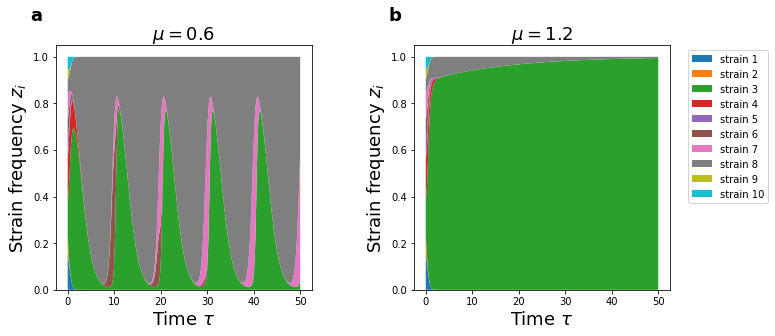}
	\caption[]{\textbf{The long time behavior can shift with changing the co-infection prevalence $\left(\mu=\frac{I}{D}\right)$  rate, when strains vary in transmission rate and transmission coefficients from mixed carriage.} We illustrate coexistence dynamics for $\mu = 0.6$ \textbf{(a)} and $\mu = 1.2$ \textbf{(b)}. We choose $\beta = 3$, $\gamma = 1.2$, $R_0 = 2$ and $b=\protect\begin{pmatrix}
	0.3 & -0.8 & 2.4 & -0.5 & 0.9 & 2 & 1.2 & 1 & -0.7 & 0.5
	\protect\end{pmatrix}.$ 
%	The matrix $\Omega$ of $\left(\omega^i_{ij}\right)_{ij}$ is
%	\begin{equation*}
%	\protect\begin{pmatrix}
%   -3 &  0 &  3&  2&  1&  2&  4& -4& -1& -5 \\
%	2 &  4 &  1& -4& -5&  3& -5& -5&  3& -1 \\
%	4 &  0 &  1& -1& -2&  4&  4& -5&  4& -4 \\
%	0 &  3 &  3&  0& 12& -1&  4&  4& -3&  1 \\
%	1 & -5 &  0&  3& -2& -2& -4& -5& -3&  1 \\
%	0 &  0 & -2&  0&  2&  2& -3& -2&  2& -5 \\
%	3 &  1 &  0& -5& -2&  1&  1&  1& -3& -4 \\
%   -3 & -4 & -2&  4&  1& -3& -1&  1& -2& -5 \\
%   -5 &  2 &  3& -3& -2&  2&  4&  1&  3& -1 \\
%	3 &  0 & -5& -1& -4&  3& -3& -4& -1&  3
%	\protect\end{pmatrix}
%	\end{equation*}
	It can be seen in this case that an  increase in $\mu$ (reducing co-infection prevalence), shifts the cycle of persistent strains in \textbf{(a)}, to the competitive exclusion of strain 3-with biggest transmission rate $\beta_i$ in \textbf{(b)}. (\href{https://github.com/lthminhthao/quasineutral_coinfection_N-strains_multiple-variations/blob/master/4_variation_in_transmission-rates-prob.ipynb}{Data \& Codes})}
	\label{A5}
\end{figure}
%%%%%%%%%%%%%%%%%%%%%%%%%%%%%%%%%%%%%%%%%%%%%%%%%%%%%%%%%%%%%%%%%%%%%%%%%%%%%%%%%%%%%%%%%%%%%%%%%%%%%%%%%%%
\subsubsection{Variation in coinfection clearance rates and transmission probability from mixed carriage, $\mathcal{A} = \{3,4\}$}\label{sec4.6}
%%%%%%%%%%%%%%%%%%%%%%%%%%%%%%%%%%%%%%%%%%%%%%%%%%%%%%%%%%%%%%%%%%%%%%%%%%%%%%%%%%%%%%%%%%%%%%%
When there are perturbations in coinfection clearance rates and transmission probability from mixed carriage, pairwise invasion fitness matrix $\Lambda$ is not anti-symmetric anymore. The analysis of the sections \ref{sec:A3} and \ref{sec:A124} suggest that, depending on the ratio $\dfrac{\Theta_4}{\Theta_3}$, we may observe  coexistence through stable steady states if $\dfrac{\Theta_4}{\Theta_3}\ll1$ and through cycles if $\dfrac{\Theta_4}{\Theta_3}\gg 1$.  
We have the explicit formula
$$\dfrac{\Theta_4}{\Theta_3} = \dfrac{2m}{\gamma} = 2\left(1 + \dfrac{r}{\gamma}\right),$$ 
then, depending on the values of $r$ and $\gamma$, we can have other interesting phenomena.\\
We make simulations for two cases of $r$, susceptible host recruitment rate. When $r = 0.2$ small enough, we obtain the coexistence of 3 strains, that is structurally stable, although it oscillates in a first period of time. When $r = 3$ large enough, the coexistence of strains becomes structurally unstable. It can be seen that, the number of coexistent strains is 3, which is odd as mentioned.
\begin{figure}[htb!]
	\centering
	\captionsetup{singlelinecheck=off}
	\includegraphics[width = 0.8 \textwidth]{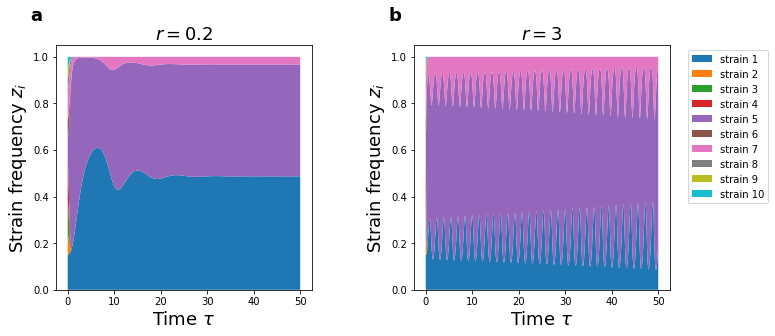}
	\caption[]{\textbf{The long time behavior can shift with changing neutral transmission rate, when strains vary in transmission rate and transmission coefficients from mixed carriage.} We illustrate coexistence dynamics for $r = 0.2$ \textbf{(a)} and $r = 3$ \textbf{(b)}. We choose $k = 3$, $R_0 = 2$, $\gamma = 1$ and reuse the initial values in figure \ref{A3}. 
%	The matrix of value's $ u_{ij}$ is chosen as follows 
%	\begin{equation*}
%	\begin{pmatrix}
%	-0.8& -4.2& -2.3& -1.9& -2.3& -2.1&  3.6&  1.8& -3.1&  1.9\\
%	-0.4&  0.3&  4.5& -0.8& -0.4&  0.9& -4.6&  0.9& -1.2& -4.6\\
%	2.1&  1.9& -4.7&  3.3&  2.3&  3. &  1.8& -4.7& -2.8&  0.2\\
%	3.5& -1.6& -1.4&  2.6& -1.1&  2.6& -3.9&  3.9&  1.5& -3.4\\
%	-4.6&  1.7&  2.2& -1.4& -3.5&  2. &  2.5& -5. 0& -3.6&  1.9\\
%	0.6&  0.3&  2.9& -4.8&  0.3&  2. &  3.7& -4.1& -4.9&  3. 0\\
%	1.1&  0.8& -0.2& -3.8& -4.4&  2.4& -0.9& -1.4&  0.6&  1.8\\
%	0.4&  3.7& -0.8&  4.9&  2.3& -3.8&  1.2&  1.6& -1. &  3.3\\
%	4.6&  0.7& -4.3&  2.0 &  4. 0 & -0.4&  0.3&  1.5&  3.3& -2.6\\
%	3.4& -4.9& -3.3& -4.3& -2.1&  2.3& -0.2& -4.4& -3.0 &  2.3
%	\end{pmatrix}.
%	\end{equation*} The matrix $\Omega$ of $\left(\omega^i_{ij}\right)_{ij}$ is
%	\begin{equation*}
%	\protect\begin{pmatrix}
%	-4.8&  0.4&  4.1&  4.6& -1.8&  4.1& -3.8& -2.2& -5.2&  1.4\\
%	-1.4&  0.8& -2.4&  3.1& -2.7& -0.8& -5.2& -4.6& -1.4& -2.8\\
%	5.1& -3.3&  1.8&  0.1& -4.7&  5.5& -3.7&  3.9&  5.4&  1.1\\
%	-4.1& -1.8&  1.7& -2.8& -5.5& -0.9& -2.2& -5.6&  5.8&  5.2\\
%	-3.3&  5.6&  4.5& -0.3&  1.8&  3.1&  3.7&  4.1&  4.7&  0.5\\
%	4.5&  0.5&  3.4& -4.2& -2.6& -3.2&  5.4&  2.3&  5. &  5. \\
%	3. & -5.4&  5.1& -3.3&  2.4& -5.6&  3.2&  3.2&  0.3& -5.9\\
%	2.6&  1.8& -5.1&  5.2& -0.8& -4.4&  5.6& -4.4& -5.3&  0.4\\
%	2.4& -5.5& -2.2& -4.7& -3.4&  1.5& -1.5&  2.8&  6. &  1.3\\
%	4. & -4.9& -0.2& -4.7& -6. & -0.3&  1.8&  2.9& -5.4&  0.9
%		\protect\end{pmatrix}.
%		\end{equation*}
		As $r$ increases, the stable state coexistence of 3 strains in \textbf{(a)}, shifts to the unstable trajectory of strains in \textbf{(b)}. (\href{https://github.com/lthminhthao/quasineutral_coinfection_N-strains_multiple-variations/blob/master/5_variation_in_coinfection-clearance-rates_transmission-prob.ipynb}{Data \& Codes})}
	\label{A6}
\end{figure}
%%%%%%%%%%%%%%%%%%%%%%%%%%%%%%%%%%%%%%%%%%%%%%%%%%%%%%%%%%%%%%%%%%%%%%%%%%%%%%%%%%%%%%%%%%%%%%%%%%%%%%%%%%%%%%%%%%%%%%%%%%%%%%%%%%%%%%%%%%%%%%%%%%%%%%%%%%%%%%%%%%%%%%%%%%%%%%%%%%%%%%%%%%%%%%%%%%%%%%%%%%%%%%%%%%%%%%%%%%%%%%%%%%%
\subsection{Summary of multi-strain outcomes by studying the slow system}
In general, when there are many traits varying among similar among similar strains, the long time behaviour may lead to complex outcomes. 
However, in cases of single trait perturbations only $\mathcal{A}=\{j\}$ , the outcome is often easier to understand.
\begin{itemize}
\item If $j\in\{1,2,3,4\}$ (the cases explored within this section), then we can proof or refer to existing result to explore the dynamics. 
In particular, in these cases, the values of the pairwize fitness $\Lambda_{i^j}$ do  not depends on the parameters of the neutral system and then  \\{\bf If $\mathcal{A}=\{j\}$ with $j\in\{1,2,3,4\}$ the qualitative outcome do not depends on the parameters of the Neutral model. }\\
\item If $j=5$ (perturbation in $k_{ij}$ only) the outcomes is more complexe and an introduction to the phenomena is given in \cite{Gjini1}. 
In particular, the pairwize fitness reads $\Lambda_{i^j}=\alpha_{jj}-\alpha_{ji}+\mu(\alpha_{ji}-\alpha_{ij}$ does depends on the parameter $\mu=\dfrac{I^*}{D^*}=\dfrac{1}{k(R_0-1)}$.
It follows that 
\\{\bf if $A=\{5\}$ the qualitative outcome do not depends on the parameters of the Neutral model.}
\end{itemize}
In the table \ref{tab:table3} we give a summary  the behavior with when there is a  perturbation in only one trait.\\

In general, when there is perturbation in several traits, the qualitative outcomes result in a complex  manner of each single case. 
The weight of each perturbation in the $\lambda_i^j$'s, and thus on the qualitative dynamics, is  govern exactly by the $\Theta_i$'s which do depend on the parameter of the neutral system.
Hence, if  the  ratio between the $\Theta_i$ is changing we may observed a change in the qualitative dynamics. Hence, a change in the parameter of the neutral models ($k$, $R_0$, $r$, $\gamma$, $\beta$) may affect not only the speed of the dynamics but also, and in a complex manner, the type of the dynamics. 
\begin{table}[htb!]
	\caption{\textbf{Summary of outcome type for each case of single traits varying}
	Note that all the $\Theta_i$ admits the same denominator $|P|=\left(2{T^*}^2 - I^*D^*\right)>0$. Since, only the ratio between the $\Theta_i$ impact  the qualitative behavior, we represent the values of $|P|\Theta_i$. 
	\label{tab:table3}}
\begin{center}
	\scalebox{1}{
		\begin{tabular}{ l p{7.5cm} p{3.5cm} p{4cm}} 
			\hline
			& \textbf{Trait varying} & $\textbf{Formula~of}$ $|P|\Theta_i$ & \textbf{Type of dynamics} \\
			\hline
			%	\textbf{Original system} &  \\
			&&&\\
			1. & Transmission rates $\beta_i $ & $ 2\beta S^*{T^*}^*$ & Competitive exclusion \\
			&&&\\
			\hline
			&&&\\
			2. & Single infection clearance rates $\gamma_{i} $ & $\gamma I^*\left(I^* + T^*\right)$ & Competitive exclusion \\ 
			&&&\\
			\hline
			&&&\\
			3. & Co-infection clearance rates $\gamma_{ij}$ & $\gamma T^*D^*$ & Possibility of Coexistence\\ 
			&&&\\
			\hline
			&&&\\
			4. & Transmission probability from mixed carriage $p_{ij}^s $ & $2m T^*D^*$ & 
			Family of cycles\\
			&&&\\
			\hline
			&&&\\
			5. & Co-colonization interaction factor via altered susceptibilities, $k_{ij}$ & $\beta T^*I^*D^*$ & Anything \\
			&&&\\
			\hline
			&&&\\
		\end{tabular}
	}
\end{center}
\end{table}

%%%%%%%%%%%%%%%%%%%%%%%%%%%%%%%%%%%%%%%%%%%%%%%%%%%%%%%%%%%%%%%%%%%%%%%%%%%%%%%%%%%%%%%%%%%%%%%%%%%%%%%%%
%%%%%%%%%%%%%%%%%%%%%%%%%%%%% Fin gros changements Madec%%%%%%%%%%%%%%%%%%%%%%%%%%%%%%%%%%%%%%%%%%%%%%%%%
%%%%%%%%%%%%%%%%%%%%%%%%%%%%%%%%%%%%%%%%%%%%%%%%%%%%%%%%%%%%%%%%%%%%%%%%%%%%%%%%%%%%%%%%%%%%%%%%%%%%%%%%%
\section{Concluding remarks}\label{sec6}
This mathematical study provides a fundamental advance in understanding analytically quasi-neutral dynamics between multiple strains in a co-infection system. Until now, explicit and general derivations of coinfection dynamics among $N$ strains are very rare in the literature \cite{adler1991dynamics,Madec2}. Previous studies have considered $N=2$, $N=3$ or $N$-strain dynamics without coinfection, typically with variation in just one fitness dimension. Others have sketched the conceptual framework linking neutrality with non-neutral dynamics \cite{Lipsitch2009}. Here, we go beyond the state of the art, and provide a full analytical characterization of the coinfection dynamics among $N$ strains that vary along multiple fitness dimensions, under the assumption that such variation is relatively small. We complete a series of studies based on slow-fast dynamics, made explicit, for linking neutral and non-neutral dynamics in interacting multi-strain pathogens \cite{Gjini,Gjini1,Madec2}. 

Naturally in this endemic compartmental model, infectious strains compete for susceptible and singly-colonized hosts, which are the only resources that can favour their growth and propagation. The different traits provide each strain with variable fitness advantages or disadvantages in exploiting such dynamic resources in the system, and interact together to shape multi-strain selection. We establish some remarkable results  by simplifying the dynamics when small perturbations arise in the clearance rates, transmission rates, within-host competitiveness coefficients, as well as co-colonization interaction factors between strains. We derive the corresponding slow-fast form for the global dynamics, the system of strain frequencies with its relevance, and provide the formal approximation for solutions having error estimates. We reduce the complexity of $N^2 + N + 1 $ equations at the origin to the $N$-equation replicator model, which reduces substantially time for computation.

Instead of studying concurrently all compartmental variables, our approach separately considers the neutral system and the perturbation components, then integrates them at the final stage. It would be possible to obtain a solution immediately for the whole emergence within perturbations in all traits. Nevertheless, such an undertaking in our view would involve many massive and complicated manipulations, and hence constitute a more difficult route than the one chosen here. This difficulty led us to the main lemma, Lemma \ref{glue_lemma}. This result enables us to integrate all particular cases for the most general problem. It only leaves us concrete special cases, with the same structure, but simpler. 

As a first step, we comprehend the neutral model and deduce the globally asymptotically stable state of variables $(S, T, I)$ by $(S^*, T^*, I^*)$, which give us a conservation law for global quantities in the co-infection system, reached in a fast time scale. The expansion theorem in \cite{Ver} plays as the first chain to acknowledges connectivity between neutrality and slow-fast system. Thanks to new variables $z_i$, denoting strain frequencies, and the new time-scale $\tau = \epsilon t$, understanding the emergent model now becomes an exploration of the so-called replicator system for $\{z_i\}_{1 \leq i \leq N}$. This derivation makes sense, in light of Tikhonov's theorem. The perturbation is consequently well approximated, which helps us to explicitly demonstrate error estimates in term $O(\epsilon)$ as well. 

Concerning the system of strain frequencies, we find out and work in the invariant set $\{u \in \mathbb{R}^{N+}: z_1 + \dots + z_N =1\}$. In general, by interpreting fitness numbers, the closing equations at each section become special instances of the same replicator system of $\{z_i\}_{ 1 \leq i \leq N}$. This enables us to study the relative dominance of strains, longtime scenarios of dynamics and other important properties. Notably it appears such a replicator system leads to the disease-free equilibrium $E_1 = (1,0,\dots,0)$ under certain conditions. This approach gives an essential and sufficient condition for the linearly asymptotically stable of trivial steady state $E_i$, in which all of the fitness numbers must be negative. Another remarkable sequel is that there is at most one $\mathcal{I}$-coexistence solution for any nonempty subset $\mathcal{I}$ given.

It is exciting to envision how this approach could be extended to other epidemiological models of multi-strain dynamics. An essential requirement is that their embedded neutral system admits a central manifold which is globally stable. The challenge would then be to identify the equations governing slow motion on this manifold in each specific model. Until now we have not considered a spatial component to the multi-strain dynamics. A further perspective is considering space and a diffusion model for the replicator equation (e.g see \cite{Bratus}). Many more extensions and model applications to data in an explicit manner should be now within reach in the near future. As argued in \cite{Madec2,gjini2021towards}, this coinfection model and its dynamics could also be translated by analogy to other biological scales, e.g. the colonization dynamics of multi-species communities \cite{Allesina5638,Grilli} or gut microbiota within host \cite{Faust}, which would open new frontiers for application, interpretation and computational tool development.

\bibliographystyle{plain}
\bibliography{references}

\end{document}